\documentclass{article}
\usepackage[latin1]{inputenc}
\usepackage{amsmath}
\usepackage{amssymb}
\usepackage{amsthm}
\usepackage{mathrsfs}
\usepackage{lscape}
\usepackage{euscript}
\usepackage{latexsym,enumerate,color}
\usepackage{authblk}

\usepackage[normalem]{ulem}

\allowdisplaybreaks

\usepackage{euscript,url}
\newtheorem{theorem}{Theorem}[section]

\newtheorem{prop}[theorem]{Proposition}

\theoremstyle{definition}

\theoremstyle{remark}

\newcommand{\Z}{{\mathbb Z}}%{{\Zset}}

\newcommand{\und}[1]{\underline{#1}}

\newcommand{\F}{{\mathbb F}}
\newcommand{\cR}{{\mathcal R}}

\newcommand{\cB}{{\mathcal B}}
\newcommand{\cC}{{\mathcal C}}

\newcommand{\half}{\frac{1}{2}}
\newcommand{\CC}{\cC}

\newcommand{\sym}{{\sf Sym }}

\newcommand{\cS}{{\mathcal S}}

\newcommand{\cA}{{\mathcal A}}

\newcommand{\mtrx}[4]{\left(\begin{array}{cc} #1 & #2 \\ #3 & #4\end{array}\right)}

\newcommand{\comment}[1]{}

\newcommand{\XX}{\mathfrak{X}}

\newcommand{\js}{{ Jordan scheme}}
\newcommand{\as}{{ association scheme}}
\newcommand{\bR}{{\mathbb R}}
\newcommand{\Sm}{\mathsf{Sym}}
\newcommand{\cJ}{{\mathcal J}}
\newcommand{\supp}{\mathsf{Supp}}

\newcommand{\fJ}{{\mathfrak K}}

\makeatletter
\let\@fnsymbol\@alph

\title{On Jordan schemes}
\author{Mikhail Muzychuk\thanks{Ben Gurion University, Beer Sheva, Israel, muzychuk@bgu.ac.il.
This author was supported by the Israeli Ministry of Absorption.}\ \ 
Sven Reichard\thanks{Dresden International University,  Dresden, Germany, sven.reichard@freenet.de}\ \
Mikhail Klin\thanks{Ben Gurion University, Beer Sheva, Israel, klin@math.bgu.ac.il}
}

\date{}

\begin{document}

\maketitle

\begin{abstract}
In 2003 Peter Cameron introduced the concept of a {\it Jordan scheme} and asked whether there exist  Jordan schemes which are not symmetrisations of coherent configurations ({\it proper} Jordan schemes). The question was answered affirmatively by the authors last year and some of the examples were presented in an essay uploaded to the arXiv. In this paper we describe several infinite series of proper Jordan schemes and present first developments in the theory of Jordan schemes - a new class of algebraic-combinatorial objects. 
\end{abstract}

\section{Introduction}\label{intro}
Motivated by the problems appearing in the theory of experimental designs, Bose and Mesner introduced in
1959 \cite{BM59} a special class of matrix algebras, known nowadays as
{\it Bose-Mesner} algebras of symmetric association schemes.
In the same year Shah published the paper \cite{S59}
where he proposed a more general idea: to replace the standard matrix
product with the Jordan product. So, in fact, he introduced the objects
called later {\it Jordan schemes} by Cameron \cite{C03}. While the ideas of Bose
and Mesner led to a new direction in algebraic graph theory called
later {\it algebraic combinatorics} by Bannai and Ito \cite{BI}, Shah's idea
was not developed at all. Only in 2004 Shah's approach was analyzed by
Bailey in her book \cite{B04} where some basic properties of Jordan
schemes were proved. She observed that the symmetrization of any
association scheme (homogeneous coherent configuration in terms of \cite{B04})
is a Jordan scheme, which led to the following question posed by
Cameron \cite{C03}: "Are there any others?". Here we give an
affirmative answer to this question by providing
several infinite series of {\it proper} Jordan schemes, i.e. those
which do not appear via symmetrization of association schemes.

Although this paper has 
an "umbilical cord" connection to \cite{KMR}, its style and notation are quite different. The paper is written as \textcolor{black}{ self-contained
as possible. The first two sections require only basic knowledge in algebra and combinatorics.  Along with this in the third section and later on it is assumed that the reader has
certain acquaintance with the classical concepts
of a coherent configuration and an association scheme. 
Recent book \cite{CP}, as well as a tutorial \cite{KRRT99}, 
are referred as reasonably helpful sources.} 

\textcolor{black}{The paper is organized as follows. In Sections~\ref{intro},~\ref{JC} we provide main definitions and develop very basic properties of Jordan configurations/schemes.}

\textcolor{black}{ In Section~\ref{SJC} we prove that a proper Jordan configuration has rank at least five and determine its algebraic structure %for the smallest rank five.}
when the bound is attained.}

In Section~\ref{WFDF} we provide a prolific construction of Jordan schemes of rank five based on the ideas of Fon-Der-Flaass~\cite{FDF}. It is shown that this construction contains an infinite series of examples that are proper on the parameter level, that is, no improper Jordan scheme could have the parameters of the constructed examples.

In Section~\ref{switching}  we give another infinite series of proper Jordan schemes obtained from improper ones by a certain switching operation. In these examples the constructed Jordan schemes may have an arbitrarily large rank.

The last section contains a proof of an auxiliary statement used in the paper.

\noindent\subsection{Notation and definitions}
Let $\Omega$ be a finite set \textcolor{black}{whose elements will be called points or vertices}, $\F$ an arbitrary field. As usual, $M_\Omega(\F)$ denotes the algebra of square matrices whose rows and columns are labeled by the elements of $\Omega$. For the sake of simplicity, we will assume that $\mathsf{char}(\F)\neq 2$, although part of the statements here are valid for even characteristic too.

For any matrix $A\in M_\Omega(\F)$ we denote by $\supp(A)$ the binary relation on $\Omega$ consisting of all pairs $(\alpha,\beta)\in\Omega^2$ satisfying $A_{\alpha,\beta}\neq 0$. 

\textcolor{black}{Given a binary relation $S\subseteq \Omega^2$, we define its graph as a pair $\Gamma:=(\Omega,S)$ where $\Omega$ is the vertex set and $S$ is the set of arcs.
%Given a binary relation $S\subseteq\Omega^2$, we denote 
The adjacency matrix\footnote{\textcolor{black}{Also referred to as the adjacency matrix of the relation $S$.}} of the graph $\Gamma$} will be denoted as $\und{S}$ . The transposed relation is denoted as $S^\top$.
For an arbitrary point $\alpha\in\Omega$ we set $S(\alpha):=\{\beta\in\Omega\,|\,(\alpha,\beta)\in S\}$. 
\textcolor{black}{Note that }
\begin{equation}\label{eq0909}
\textcolor{black}{(\und{S})_{\alpha,\beta}\neq 0\iff (\und{S})_{\alpha,\beta} = 1 \iff \beta\in S(\alpha) \iff  \alpha\in S^\top(\beta).}
\end{equation}

\textcolor{black}{ Note that in general $\Gamma$ is a directed graph. It  becomes undirected if and only if $S=S^\top$. A relation $S$ is called {\it regular} if its graph $\Gamma = (\Omega,S)$ is regular, i.e. the number $|S(\alpha)|$ does not depend on the choice of $\alpha\in\Omega$. In this case $|S(\alpha)|$ is called the {\it valency} of $S$ (or $\Gamma$) and is denoted as $k_S$. In what follows we set $1_\Omega:=\{(\omega,\omega)\,|\,\omega\in\Omega\}$. The adjacency matix of $1_\Omega$ is the identity matrix denoted by $I_\Omega$.}

If $\cS$ is a collection of binary relations, then 
$\und{\cS}:=\{\und{S}\,|\,S\in\cS\}$ and  $\F\langle\und{\cS}\rangle$ stands for the linear span of the set $\und{\cS}$. Given a partition $\cS$ of $\Omega^2$, we denote by $\cS(\alpha,\beta)$ the unique class of $\cS$ which contains 
the pair $(\alpha,\beta)\in\Omega^2$.

Given two matrices $A,B\in M_\Omega(\F)$, we denote their standard matrix product as $A\cdot B$ or just $AB$; Schur-Hadamard (component-wise) product is written as $A\circ B$ and $A^\top$ stands for the matrix transposed to $A$. Note that the identity matrix $I_\Omega$ is a $\cdot$-unit while the all-one matrix $J_\Omega$ is a $\circ$-unit. We also define the \textit{Jordan} product in $M_\Omega(\F)$ as $A\star B = \half (A\cdot B + B\cdot A)$. 

\textcolor{black}{In what follows we will often multiply adjacency matrices of various binary relations. In these cases the following simple statement is very useful. 
\begin{prop}\label{090920a}Let $\und{R},\und{S}\in M_\Omega(\F)$
be adjacency matrices of binary relations $R,S\subseteq \Omega^2$.
Then $(\und{R}\cdot\und{S})_{\alpha,\beta} = |R(\alpha)\cap S^\top(\beta)|$ if $\mathsf{char}(\F) = 0$ and $(\und{R}\cdot\und{S})_{\alpha,\beta}\equiv |R(\alpha)\cap S^\top(\beta)|\ ({\rm mod}\ \mathsf{char}(\F))$ if  $\mathsf{char}(\F) > 0$. 
\end{prop}
\begin{proof} Since $\und{R}$ and $\und{S}$ are $(0,1)$-matrices, the $(\alpha,\beta)$-entry of the product $\und{R}\cdot \und{S}$ is equal to the number of points $\gamma\in\Omega$ satisfying $(\und{R})_{\alpha,\gamma} = 1 = (\und{S})_{\gamma,\beta}$. According to \eqref{eq0909}, these equalities are equivalent to 
$\gamma\in R(\alpha)\cap S^\top(\beta)$. Therefore 
$(\und{R}\cdot \und{S})_{\alpha,\beta} =\sum_{\gamma\in R(\alpha)\cap S^\top(\beta)} 1$. Now the claim follows. 
\end{proof}}

To avoid an excessive use of brackets we agree that $\cdot$ and $\star$ have a higher priority than $\circ$. For example, an expression $A\cdot B\circ C\star D$ should be read as $(A\cdot B)\circ (C\star D)$.

The $k$-th power of a matrix $A$ with respect to the products $\cdot,\circ$ and $\star$ will be written as $A^k, A^{\circ k}$ and $A^{\star k}$, respectively.

The algebra $(M_\Omega(\F),\star)$ is a particular case of algebraic systems known as Jordan algebras \cite{J}. Recall that a {\it Jordan algebra} over the field $\F$ is a vector space $\cA$ over $\F$, supplied with a bilinear multiplication $\star$ satisfying
the following axioms:
\begin{equation}\label{JA}
\begin{array}{rcl}
a\star b & = & b\star a;\\
(a\star b)\star (a\star a) & = & a\star (b \star (a\star a)).
\end{array}
\end{equation} 
Given an associative algebra $(\cA,\cdot)$, the derived product 
$a\star b :=\frac{1}{2}(a\cdot b + b\cdot a)$ satisfies the above axioms and produces a Jordan algebra. A Jordan algebra is called {\it special} \cite{M} if it is a $\star$-subalgebra of a Jordan algebra obtained in this way. If $\cA$ is an associative algebra, then the corresponding special Jordan algebra will be denoted as
$\cA^+$. 
%If an associative algebra $(\cA,\cdot)$ admits an automorphism or anti-automorphism $\alpha$, then the subspace $\cA^\alpha:=\{a\in\cA\,|\, a^\alpha=a\}$ is a Jordan subalgebra of $\cA^+$. 
%$\{a\in\cA\,|\, a^\alpha=a\}$ is a Jordan subalgebra of $\cA^+$. 
The non-special Jordan algebras are called {\it exceptional}. Note that all Jordan algebras appearing in this paper are subalgebras of $M_\Omega(\F)^+$, and, therefore, are special.

To distinguish isomorphisms of associative and Jordan algebras we use the notation $\cong_J$ for an isomorphism between Jordan algebras.

\noindent\subsection{Subalgebras of the matrix algebra}
Recall that a vector subspace 
$\cA\subseteq M_\Omega(\F)$ is called a \textit{coherent algebra} \cite{H87} if it contains $I_\Omega,J_\Omega$ and is closed w.r.t. ${}^\top,\cdot,\circ$. Similarly, a \textit{coherent Jordan algebra} (coherent J-algebra, for short) is a subspace $\cA\subseteq M_\Omega(\F)$ which satisfies the same conditions where $\cdot$ is replaced by $\star$. Clearly, each coherent algebra is a coherent J-algebra. The converse is not true - the simplest example of a coherent J-algebra which is not a coherent algebra is provided by the subspace $\sym_\Omega(\F)$ of symmetric matrices, $|\Omega| > 1$. 

Given a coherent algebra $\cA$, its \textit{symmetrization} 
$\widetilde{\cA}:=\{A\in\cA\,|\, A^\top = A\}$ is a coherent J-algebra. It is a coherent algebra iff the matrices of $\widetilde{\cA}$ pairwise commute.

We start with the following statement.
\begin{prop}\label{020619a} Let $\cA\subseteq M_\Omega(\F)$ be a $\circ$-closed subspace (Schur-Hadamard subalgebra). Then there exists a unique basis $A_1,...,A_r$ of $\cA$ consisting of $\{0,1\}$-matrices such that $\supp(A_i)\cap\supp(A_j)=\emptyset$ whenever $i\neq j$. Moreover, 
\begin{enumerate}
\item[{\rm (a)}] if $I_\Omega\in\cA$, then $I_\Omega = \sum_{i\in F} A_i$  for some $F\subseteq\{1,...,r\}$;
\item[{\rm (b)}] if $J_\Omega\in\cA$, then  $\sum_{i=1}^r A_i = J_\Omega$;
\item[{\rm (c)}] if $\cA^\top = \cA$, then $\{A_1,...,A_r\}^\top = \{A_1,...,A_r\}$. 
\end{enumerate}
\end{prop}
\begin{proof} The algebra $(M_\Omega(\F),\circ)$ is a commutative associative algebra isomorphic to $\F^{|\Omega^2|}$. Its subalgebra $\cA$ is isomorphic to $\F^r,r=\dim(\cA)$ and has a basis consisting of pairwise orthogonal $\circ$-idempotents, say $A_1,...,A_r$. It follows from $A_i\circ A_i = A_i$ that each $A_i$ is a $\{0,1\}$-matrix. If $i\neq j$, then $A_i\circ A_j =0$ implying $\supp(A_i)\cap \supp(A_j)=\emptyset$. 

\medskip

\noindent (a) It follows from $I_\Omega\in\cA$ that $I_\Omega = \sum_i c_i A_i$ for some $c_i\in\F$. 
\textcolor{black}{Taking $\circ$-square of }both sides we obtain $I_\Omega = \sum_i c_i^2 A_i$. Therefore $c_i\in\{0,1\}$, as required. 

\medskip

\noindent (b) As in the previous part one can show that 
$J_\Omega$ is a $\{0,1\}$-linear combination of the standard basis matrices:
$J_\Omega = \sum_i c_i A_i, c_i\in\{0,1\}.$ Pick an arbitrary $i$ and $(\alpha,\beta)\in\supp(A_i)$. Then $1 = (J_\Omega)_{\alpha,\beta} = c_i$. Hence all coefficients  in the above decomposition are ones.

\medskip

Part (c) follows directly from the fact that ${}^\top$ is an automorphism  of $(\cA,\circ)$. 
\end{proof}
In what follows the above basis will be called the \textit{standard basis} of $\cA$.
 
Since every $\{0,1\}$-matrix $A\in M_\Omega(\F)$ is the adjacency matrix of $\supp(A)$, one has the following consequence of Proposition~\ref{020619a}.
\begin{prop}\label{020619z} Let $\cA\subseteq M_\Omega(\F)$ be 
a subspace which contains $I_\Omega,J_\Omega$ and is closed w.r.t. ${}^\top$ and $\circ$. Then there exists a unique partition $\cC=\{C_1,...,C_r\}$ of $\Omega^2$ such that $\cA = \F\langle \und{C}_1,...,\und{C}_r\rangle$. The partition satisfies the following conditions 
\begin{enumerate}
\item[{\rm (a)}] $1_\Omega = \{(\omega,\omega)\,|\,\omega\in\Omega\}$ is a union of some $C_i\in\cC$;
\item[{\rm (b)}] $\cC^\top = \cC$\textcolor{black}{, where $\CC^\top:=\{C_1^\top,...,C_r^\top\}$}. 
\end{enumerate}
\end{prop}
Following \cite{H90}, we call any partition $\CC$ of $\Omega^2$ satisfying the conditions (a)-(b) a \textit{rainbow}.
\textcolor{black}{ The numbers $|\Omega|$ and $|\CC|$ are called the {\it order} and the {\it rank} of the rainbow.
A rainbow $(\Omega,\CC)$ has rank one if and only if $|\Omega|=1$. If the point set  $\Omega$ has at least two elements, then $|\CC|\geq 2$ where the equality holds if and only if $\CC=\{1_\Omega,\Omega^2\setminus 1_\Omega\}$. These rainbows
are called {\it trivial}.
}

We call the elements of $\CC$ \textit{basic relations} or \textit{color classes} of $\cC$ (and $\cA$). The corresponding graphs $(\Omega,C), C\in\cC$ are called the \textit{basic graphs} of the rainbow $\XX=(\Omega,\cC)$. Any union of basic relations of $\XX$ is called an $\XX$(or $\CC$)-relation.
The set of all $\CC$-relations is denoted as $\CC^\cup$. \textcolor{black}{The difference between basic and non-basic relations of $\CC$ is vividly seen in the following statement: given a basic relation $C\in\CC$ and an arbitrary $D\in\CC^\cup$, either $C\cap D=\emptyset$ or $C\subseteq D$. In the paper we  will often use this simple observation without referring.}

A convenient way to present a rainbow (or an arbitrary partition $\CC$ of $\Omega^2$) is via its {\it adjacency matrix} 
\textcolor{black}{${\mathfrak A}(\CC)$ obtained by assigning  ${\mathfrak A}(\CC)_{\alpha,\beta}:=\CC(\alpha,\beta)$.
%For example, the adjacency matrix of the rainbow with point set $\Omega = \{1,2,3\}$ and basic relations
%$$\CC=\left\{\begin{array}{c} A=\{(1,1)\},B=\{(1,2),(1,3)\},
%C=\{(2,1),(3,1)\},\\
%D=\{(2,2),(3,3)\},E=\{(2,3),(3,2)\}
%\end{array}\right\}$$ has the following form
%$$
%\mathfrak{A}=\left(
%\begin{array}{cccc}
%A & B & B \\
%C & D & E \\
%C & E & D
%\end{array}
%\right)
%$$
 }

 Note that every basic relation  $C\in\cC$ is either {\it symmetric} ($C^\top = C$) or {\it anti-symmetric} ($C^\top\cap C=\emptyset$). A rainbow is called \textit{symmetric} if all its basic relations are symmetric. A rainbow 
will be called \textcolor{black}{\textit{regular}} if all its basic relations are regular. \textcolor{black}{ It is called {\it homogeneous} if $1_\Omega$ is a basic relation of a rainbow. It is easy to see that a regular rainbow is always homogeneous while the opposite is not true.}

Two rainbows $(\Omega,\CC)$ and $(\Omega',\CC')$ are called \textit{(combinatorially) isomorphic} if there exists a bijection $f:\Omega\rightarrow\Omega'$ which maps the coloring $\CC$ onto $\CC'$ bijectively. 

A rainbow $(\Omega,\CC)$ is called a \textit{coherent configuration} (CC, for short) \cite{H75} if 
it satisfies the following regularity condition:
\begin{equation}\label{eq1}
\forall C,D\in\cC\ 
\forall\alpha,\alpha',\beta,\beta'\in\Omega:\ \cC(\alpha,\beta)=\cC(\alpha',\beta')\implies  |C(\alpha)\cap D^\top(\beta)| =  |C(\alpha')\cap D^\top(\beta')|
\end{equation}
In other words, the cardinality of $C(\alpha)\cap D^\top(\beta)$ depends only on the color class of the pair $(\alpha,\beta)$. The numbers $p_{C,D}^F:=|C(\alpha)\cap D^\top(\beta)|$, where $F:=\CC(\alpha,\beta)$, are called the \textit{intersection numbers} of the CC $(\Omega,\cC)$. 

The statement below describes a well-known relationship between 
coherent configurations and coherent algebras \cite{H87}.
\begin{theorem}\label{020619b}
Let $\XX=(\Omega,\cC)$ be a rainbow and $\F$ a field of characteristic zero. The vector subspace $\F\langle\und{\cC}\rangle\subseteq M_\Omega(\F)$ is a coherent algebra if and only if $\XX$ is a coherent configuration.
Every coherent algebra $\cA\subseteq M_\Omega(\F)$ coincides with $\F\langle\und{\CC}\rangle$ for a uniquely determined CC $(\Omega,\CC)$.
\end{theorem}
Thus in the case of $\mathsf{char}(\F)=0$ there is a one-to-one correspondence between the coherent subalgebras of $M_\Omega(\F)$ and CCs over $\Omega$. Given a CC $\XX=(\Omega,\cC)$, the linear span $\F\langle \und{\cC}\rangle$ is a coherent algebra which is often called the \textit{adjacency algebra} of $\XX$. Its 
standard basis coincides with $\{\und{C}\}_{C\in\cC}$. The intersection numbers appear as structure constants of the adjacency algebra w.r.t. the standard basis: $\und{C}\cdot\und{D} = \sum_{F\in\cC} p_{C D}^F \und{F}$. 

\textcolor{black}{A homogeneous coherent configuration is called an {\it association scheme} (AS, for short). An association scheme is called {\it commutative} if its adjacency algebra is commutative. An AS is called {\it symmetric} if all its relations are symmetric. Note that a symmetric AS is always commutative.}

Regarding coherent J-algebras we have a complete analogue of  Theorem~\ref{020619b}. 
\begin{theorem}\label{040619a} Let $\cA\subseteq M_\Omega(\F), \mathsf{char}(\F)=0$ be a coherent J-algebra. Then there exists a rainbow $\cC$ of $\Omega^2$ such that $\{\und{C}\,|\,C\in\cC\}$ is the uniquely determined standard basis of $\cA$. \textcolor{black}{The rainbow $(\Omega,\CC)$ satisfies the following condition:
\begin{equation}\label{eq2}
\forall C,D\in\cC\ 
\forall\alpha,\alpha',\beta,\beta'\in\Omega:\ \cC(\alpha,\beta)=\cC(\alpha',\beta')\implies 
$$
$$
  |C(\alpha)\cap D^\top(\beta)| + |D(\alpha)\cap C^\top(\beta)|=  |C(\alpha')\cap D^\top(\beta')|+|D(\alpha')\cap C^\top(\beta')|.
\end{equation}
}

\textcolor{black}{Vice versa, if a rainbow $(\Omega,\cC)$ satisfies \eqref{eq2}, then the linear span $\F\langle\und{\CC}\rangle$ is a coherent J-algebra provided that $\mathsf{char}(\F)\neq 2$.}
\end{theorem}
\begin{proof}
\textcolor{black}{ 
If $\cA\subseteq M_\Omega(\F)$ is a coherent J-algebra, then by Proposition~\ref{020619z} there exists a rainbow $\CC=\{C_1,...,C_r\}$ such that $\cA=\F\langle\und{\CC}\rangle$. To prove \eqref{eq2} we pick arbitrary $C,D\in\CC$ and two pairs $(\alpha,\beta),(\alpha',\beta')\in\Omega^2$ sharing the same color of $\CC$, say $F$, that is $\CC(\alpha,\beta)=\CC(\alpha',\beta') = F$.
}

\textcolor{black}{
Since $\F\langle \und{\CC}\rangle$ is $\star$-closed, $\und{C}\star\und{D} = \sum_{R\in\CC} \lambda_R \und{R}$ for some scalars
$\lambda_R\in\F$. It follows from $(\alpha,\beta)\in F \ni (\alpha',\beta')$ that $(\und{C}\star\und{D})_{\alpha,\beta} =\lambda_F =(\und{C}\star \und{D})_{\alpha',\beta'}$. Therefore
\begin{equation}\label{eq0909a}
(\und{C}\cdot\und{D})_{\alpha,\beta} + (\und{D}\cdot\und{C})_{\alpha,\beta}=(\und{C}\cdot \und{D})_{\alpha',\beta'} 
+ (\und{D}\cdot \und{C})_{\alpha',\beta'}.\end{equation}}

\textcolor{black}{
Now the property~\eqref{eq2} follows from Proposition~\ref{090920a}.
}

\textcolor{black}{
Assume now that $(\Omega,\CC)$ is a rainbow satisfying~\eqref{eq2}.
It follows from the definition of a rainbow that the vector space $\F\langle\und{\CC}\rangle$ contains $I_\Omega, J_\Omega$ and is closed with respect to $\circ$ and ${}^\top$. It remains to show that $\F\langle\und{\CC}\rangle$ is $\star$-closed. Pick two arbitrary elements of the standard basis $\und{C},\und{D}$, where $C,D\in\CC$. We have to show that 
$\und{C}\star\und{D} = \frac{1}{2}(\und{C}\cdot\und{D}+\und{D}\cdot\und{C})$ is a linear combination of $\und{F},F\in\CC$.   It follows from~\eqref{eq2} and Proposition~\ref{090920a} that~\eqref{eq0909a}
holds whenever the pairs $(\alpha,\beta)$ and $(\alpha',\beta')$ belong to the same color class of $\CC$. Thus
$\und{C}\star\und{D}$ is a linear combination of the adjacency matrices $\und{F},F\in\CC$, and, therefore, belongs to $\F\langle\und{\CC}\rangle$.
}
\end{proof}

A rainbow satisfying the above condition~\eqref{eq2} will be called a \textit{ coherent Jordan configuration} or just a \textit{Jordan configuration} (CJC or JC, for short). For homogeneous Jordan configurations we will use the name \textit{Jordan schemes} proposed by Cameron. 

The numbers $p_{C,D}^F:=\frac{1}{2}(|C(\alpha)\cap D^\top(\beta)|+|D(\alpha)\cap C^\top(\beta)|),F:=\CC(\alpha,\beta)$ are called the \textit{intersection numbers} of the Jordan configuration $(\Omega,\cC)$. Note that the intersection numbers of a JC are non-negative rational numbers. Although they might be non integral, they are always half-integral, i.e. belong to $\frac{1}{2}\Z$.

Thus in the case of $\mathsf{char}(\F)=0$ there is one-to-one correspondence between the coherent Jordan subalgebras of $M_\Omega(\F)^+$ and JCs \textcolor{black}{with point set} $\Omega$. Given a JC $\XX=(\Omega,\cC)$, the linear span $\F\langle\und{\cC}\rangle$ is a coherent J-algebra which we call the \textit{adjacency algebra} of $\XX$. Its 
standard basis coincides with $\{\und{C}\}_{C\in\cC}$. The intersection numbers \textcolor{black}{defined by~\eqref{eq2}} coincide with the structure constants of the adjacency algebra $\F\langle\und{\cC}\rangle$ w.r.t. the standard basis: $\und{C}\star\und{D} = \sum_{F\in\cC} p_{C D}^F \und{F}$. 

Two Jordan configurations $(\Omega,\CC)$ and $(\Omega',\CC')$ are
\textit{combinatorially} isomorphic if they are isomorphic as rainbows. We also say that they are \textit{algebraically} isomorphic if there exists a bijection 
$\varphi:\CC\rightarrow\CC'$ which preserves the structure constants,
that is the equality $p_{C,D}^F = p_{C^\varphi,D^\varphi}^{F^\varphi}$ holds for any triple of basic relations $C,D,F\in\CC$.

\subsection{Coherent and Jordan closures}
It follows from the definition of coherent algebra that the intersection of any number of coherent algebras is a coherent algebra too. This allows us to define a
\textit{coherent closure} of any matrix set $X\subseteq M_\Omega(\F)$ as the intersection of all coherent algebras containing $X$. The first efficient algorithm computing the coherent closure was proposed by  Weisfeiler and Leman in \cite{WL}. This algorithm used nowadays in different modifications is referred to as \textit{WL-algorithm} or \textit{WL-stabilization procedure}. The latest results regarding the complexity issues of the WL-algorithm are presented in \cite{LPS}. In what follows we denote the coherent closure of a set $X\subseteq M_\Omega(\F)$ as $WL(X)$. Note that the ideas presented in \cite{WL} were developed later in the book \cite{W76}.

In a similar way one can define the \textit{Jordan closure} $J(X)$ of any set $X\subseteq M_\Omega(\F)$ as the intersection of all coherent Jordan algebras containing $X$. To compute $J(X)$ one can modify the WL-stabilization procedure by replacing the  standard  matrix multiplication with the Jordan one (this process was called by Cameron the \textit{Jordan stabilization} ). It follows from the definitions that the following inclusion always holds: $J(X)\subseteq WL(X)$. The inclusion could be proper. For example, $\sym_\Omega(\F)$ is a coherent J-algebra but its coherent closure coincides with $M_\Omega(\F)$.

If each matrix from the set $X\subseteq M_\Omega(\F)$ is symmetric, then $J(X)$ contains symmetric matrices only, and, therefore,  $J(X)\subseteq \widetilde{WL(X)}$.  
%Cameron asked in \cite{C03} whether the inclusion could be proper.
The statement below tells us when the equality holds. To formulate it 
we recall that 
a coherent J-algebra $\cJ$ (and the corresponding Jordan configuration) is called \textit{proper} if 
 it is not the symmetrization of a coherent algebra (otherwise, we call $\cJ$ \textit{improper} or \textit{non-proper}). 
\begin{prop}\label{081219a} Let $X\subseteq M_\Omega(\F)$ be a set of symmetric matrices. Then $J(X) = \widetilde{WL(X)}$ if and only if $J(X)$ is non-proper.
\end{prop}
\begin{proof} If $J(X) = \widetilde{WL(X)}$, then $J(X)$ is non-proper just by definition. To prove the converse implication let us assume that 
$J(X) = \widetilde{\cA}$ for some coherent algebra $\cA$ (i.e. $J(X)$ is non-proper). Then $X\subseteq \cA$ implying $X\subseteq WL(X)\subseteq \cA$. Therefore, $X\subseteq \widetilde{WL(X)}\subseteq \widetilde{\cA}$ implying 
$J(X)\subseteq \widetilde{WL(X)}\subseteq \widetilde{\cA}=J(X)$. Thus
$J(X) = \widetilde{WL(X)}$, as claimed.
\end{proof}
\noindent {\bf Remark.} Since $J(\cJ)=\cJ$ holds for every coherent J-algebra $\cJ$, the above statement implies that 
$\cJ$ is proper if and only if $\cJ\neq \widetilde{WL(\cJ)}$. 
This observation fully correlates with the discussion of the connections between two closures  in  \cite{C}.
In sections~\ref{WFDF} and \ref{switching} we present first examples of proper Jordan schemes.

We conclude this subsection by 
the statement which will be referred to as {\it Schur-Wielandt principle}. It is very useful in computing of coherent and Jordan closures. 

\begin{prop}\label{300819a} Let $\cA\subseteq M_\Omega(\F)$ be 
a subspace which is closed w.r.t. $\circ$. Then for any matrix $A\in\cA$ and $c\in\F$ the adjacency matrix of the relation $\{(\alpha,\beta)\in\Omega^2\,|\, A_{\alpha,\beta} = c\}$, which we denote by $A_c$, belongs to $\cA$.
\end{prop}
\noindent{\bf Proof} follows immediately from Proposition~\ref{020619a}. \hfill $\square$

\section{Basic facts about coherent Jordan configurations}\label{JC}

Jordan configurations share some basic properties of usual coherent configurations. In this section we need only a part of them. To prove those properties we introduce the Jordan product of binary relations $R\star S := RS\cup SR$ where $RS$ is the standard relational product 
\textcolor{black}{(i.e. the composition of relations)}. We use the same notation $\star$ both for relational and matrix Jordan products because of the following identity: ${\mathsf {Supp} }(\und{R}\star \und{S}) = R\star S$.  Since $\cC$ is a JC, the set $\cC^\cup$ is closed with respect to the boolean operations and the Jordan relational product $\star$. 

The statement below collects
the main properties needed for further presentation.
\begin{prop}\label{040619b} Let $\XX=(\Omega,\cC)$ be a JC. Then 
\begin{enumerate}
\item[{\rm (a)}] there exists a partition $\Omega_1,...,\Omega_f$ of $\Omega$ such that $1_{\Omega_i}\in \cC$, the sets $\Omega_i$ are called the \textit{fibers} of $\XX$;
\item[{\rm (b)}] given a basic relation $C\in\cC$, there exist two fibers $\Omega_i,\Omega_j$ (not necessarily distinct) s.t. 
$C\subseteq (\Omega_i\times \Omega_j) \cup (\Omega_j\times \Omega_i)$;
if $\omega,\omega'\in\Omega_i$ or $\omega,\omega'\in\Omega_j$, then $|C(\omega)|+|C^\top(\omega)|=|C(\omega')|+|C^\top(\omega')|$. In opther words, the graph $(\Omega_i\cup\Omega_j, C\cup C^\top)$ is bi-regular;
\item[{\rm (c)}] $\cC$ is a disjoint union $\cC = \bigcup_{1\leq a \leq b \leq f} \cC^{ab}$ where $\cC^{ab} = \{C\in\cC\,|\, C\subseteq\Omega_a\times\Omega_b\cup \Omega_b\times\Omega_a\}$.
\end{enumerate}
\end{prop}
\begin{proof}  The first property follows from the definition of a rainbow.

(b)+(c) As $\cC^\cup$ is $\star$-closed, the product 
$(1_{\Omega_i}\star \Omega^2)\star 1_{\Omega_j}$ belongs to $\cC^\cup$ for any pair of indices $1\leq i,j\leq f$. If $i\neq j$, then $(1_{\Omega_i}\star \Omega^2)\star 1_{\Omega_j} = (\Omega_j\times\Omega_i)\cup (\Omega_i\times\Omega_j)$. Therefore  the sets $(\Omega_j\times\Omega_i)\cup (\Omega_i\times\Omega_j), 1\leq i < j\leq f$ belong to $\cC^\cup$. Now it follows from 
$$
\cC^\cup\ni 
(1_{\Omega_i}\star \Omega^2)\setminus \left( \bigcup_{j\neq i}(\Omega_i\times\Omega_j)\cup (\Omega_j\times \Omega_i) \right) =\Omega_i\times\Omega_i
$$
that any relation of the form $\Omega_i\times\Omega_j\cup \Omega_j\times \Omega_i$ (here we do not exclude the case of $i=j$) belongs to $\cC^\cup$. Those relations form a partition of $\Omega^2$. Therefore each basic relation $C$ is contained only in one of them. Thus $C\subseteq (\Omega_i\times\Omega_j)\cup (\Omega_j\times\Omega_i)$ for some $i,j$. \textcolor{black}{ Pick an arbitrary pair $\omega,\omega'\in\Omega_i$.
Then $\CC(\omega,\omega)=1_{\Omega_i} =\CC(\omega',\omega')$ and it follows from~\eqref{eq2} with $D=C$
 that }
$|C(\omega)| + |C^\top(\omega)|=|C(\omega')| + |C^\top(\omega')|$. This proves the second part of the claim.

Part (c) follows directly form the first part of (b).
\end{proof}

\textcolor{black}{Let $\Delta = \bigcup_{i\in I} \Omega_i$ be a non-empty union of fibers of $\XX$. It follows from the above proposition that the set of relations $\CC_\Delta: = \bigcup_{i,j\in I} \CC^{ij}$ forms a Jordan configuration on $\Delta$. In particular, the set of relations 
$\CC^{ii}$ forms a homogeneous  JC on a fiber $\Omega_i$.  As in the theory of CCs we call $(\Omega_i,\CC^{ii})$ a {\it homogeneous constituent} of $(\Omega,\CC)$.}

It is well-known that a coherent configuration is homogeneous if and only if \textcolor{black}{it is regular.} The following example shows that for Jordan configurations it is not true anymore. 

Let $\XX$ be a rank $4$ rainbow on the point set $\Omega = \{1,2,3,4\}$ the adjacency matrix of which has the following 
form
\textcolor{black}{$$
{\mathfrak A}:=\left(
\begin{array}{cccc}
X& Y & Z & Z\\
Y & X & Z & Z\\
W & W & X & Y \\
W & W & Y & X
\end{array}
\right).
$$
}
A direct check shows that $\XX$ is a \textcolor{black}{homogeneous but  non-regular Jordan configuration.} The statement below describes when such a situation occurs.
\begin{prop}\label{240719a} Let $\XX=(\Omega,\CC)$ be a JC with one fiber. If $\XX$ is not homogeneous, then there exists a bi-partition $\Omega = \Omega_0\cup \Omega_1$ with
$|\Omega_1|=|\Omega_0|$ such that each non-regular $C\in\CC$ is \textcolor{black}{"bi-regular"}, i.e. $C\subseteq \Omega_i\times\Omega_{1-i}$ for some $i\in\{0,1\}$ and $|C(\omega)|$ is constant for all $\omega\in\Omega_i$. 
\end{prop}
\begin{proof} Denote by $\cA$ the adjacency algebra of $\XX$ over the rationals. Pick an arbitrary $C\in\CC$. Consider the product $\und{C}\star J$ (here and later in the proof $J:=J_{\Omega}$). By direct calculations we obtain that $(\und{C}\star J)_{\alpha,\beta} =\half( |C(\alpha)| + |C^\top(\beta)|)$. Since $1_\Omega$ is a basic relation, every matrix in $\cA$ has a constant diagonal. Therefore, $|C(\alpha)|+|C^\top(\alpha)|$ does not depend on $\alpha$. Let us denote this number as $k$. Then  
$(\und{C}\star J)_{\alpha,\beta} =\frac{1}{2}(|C(\alpha)|-|C(\beta)|+k)$. Assume now that $C$ is not regular. Denote by $\Omega_0,\Omega_1$ the subsets with the maximal and minimal values of $|C(\omega)|$, respectively. Since $C$ is not regular, the sets $\Omega_0,\Omega_1$ are non-empty and disjoint. \textcolor{black}{It follows from the definition of the $\Omega_i$'s that the function  $\omega\mapsto |C(\omega)|$ is constant on $\Omega_0$ and $\Omega_1$. We set $k_i:=|C(\omega)|,\omega\in\Omega_i$. 
It follows from the definition of $\Omega_0$ and $\Omega_1$ that  $k_0 > k_1$.} 

It follows from the choice of $\Omega_i,i=0,1$ that the maximal value of the matrix entries $(\und{C}\star J)_{\alpha,\beta}$ is reached iff $(\alpha,\beta)\in\Omega_0\times \Omega_1$.  By the Schur-Wielandt principle $\und{\Omega_0\times\Omega_1}\in\cA$ and, therefore, 
$\Omega_1\times\Omega_0\in\cC^\cup$  and $\cC^\cup\ni (\Omega_0\times\Omega_1)\star (\Omega_1\times\Omega_0)\cap 1_\Omega = 1_{\Omega_0}\cup 1_{\Omega_1}$. Together with $1_\Omega\in \cC$ we conclude that 
$1_{\Omega_0}\cup 1_{\Omega_1} = 1_\Omega$, or, equivalently, $\Omega_0\cup \Omega_1 = \Omega$. \textcolor{black}{Therefore 
the sets $\Omega_0,\Omega_1$ form the required bi-partition of $\Omega$.}
%It follows from the definition of $\Omega_0$ and $\Omega_1$ that $|C(\omega)|$ depends only on the part $\Omega_i$ to which $\omega$ belongs.

\textcolor{black}{Let us show now that $|\Omega_0|=|\Omega_1|$. It follows from $\und{\Omega_0\times\Omega_1}\in\cA\ni \und{\Omega_1\times\Omega_0}$ that  $$
(\und{\Omega_0\times\Omega_1})\star
(\und{\Omega_1\times\Omega_0})  \in\cA.$$
Direct computation yields us
$$
(\und{\Omega_0\times\Omega_1})\star
(\und{\Omega_1\times\Omega_0}) =
%\frac{1}{2}\left( (\und{\Omega_0\times\Omega_1})\cdot
%(\und{\Omega_1\times\Omega_0})+
%(\und{\Omega_1\times\Omega_0})\cdot
%(\und{\Omega_0\times\Omega_1})\right) = 
\frac{1}{2}\left(|\Omega_1|(\und{\Omega_0\times\Omega_0})+|\Omega_0|(\und{\Omega_1\times\Omega_1})\right)\in\cA.$$
If $|\Omega_0|\neq |\Omega_1|$, then by the Schur-Wielandt principle
$\Omega_0\times \Omega_0\in\CC^\cup$ implying 
$1_{\Omega_0} =(\Omega_0\times \Omega_0)\cap 1_\Omega\in \CC^\cup$ contrary to the assumption that $\Omega$ is a fiber. 
Therefore $|\Omega_0|=|\Omega_1|$.}

\textcolor{black}{Now we prove that $C\subseteq \Omega_{i}\times\Omega_{1-i}$ for some $i=0,1$. \\
Since $\Omega_0\times\Omega_1$ and 
$\Omega_1\times\Omega_0$ are $\CC$-relations, the set  
$\Omega^2\setminus (\Omega_0\times\Omega_1\cup \Omega_1\times\Omega_0) = \Omega_0^2\cup\Omega_1^2$ is a $\CC$-relation too.
Since $C$ is a basic relation, it is contained only in one of the three relations: $\Omega_0\times\Omega_1,\Omega_1\times\Omega_0,  \Omega_0^2\cup\Omega_1^2$. If $C$ is contained in the first or second relation, then we are done. Assume, towards a contradiction, that $C\subseteq\Omega_0^2\cup\Omega_1^2$. Then
$|C\cap\Omega_0^2| = \sum_{\alpha\in\Omega_0}|C(\alpha)| =
k_0 |\Omega_0|$. On the other hand,
$$
|C\cap\Omega_0^2| = |(C\cap\Omega_0^2)^\top| = 
|C^\top\cap\Omega_0^2| = \sum_{\alpha\in\Omega_0}|C^\top(\alpha)| = \sum_{\alpha\in\Omega_0}(k - |C(\alpha)|) =
(k-k_0) |\Omega_0|.$$ Therefore $k=2k_0$. The same counting applied to $|C\cap \Omega_1^2|$ yields us $k=2k_1$. Hence $k_0 = k_1$, contrary to $k_0 > k_1$.}

It remains to show that the partition $\Omega = \Omega_0\cup\Omega_1$ does not depend on the choice of a non-regular  
relation $C\in\CC$. Consider another non-regular relation $C'\in\CC$. Let $\Omega'_1$ and $\Omega'_0$ be the parts of the bi-partition corresponding to $C'$. Assume that the partition $\{\Omega_0,\Omega_1\}$ is different from $\{\Omega'_0,\Omega'_1\}$ (recall that $|\Omega_i|=|\Omega'_i|=|\Omega|/2$ for $i=0,1$).
Then $\Omega_i\cap\Omega'_j \neq \emptyset$ for any $i,j\in\{0,1\}$.
Since $\Omega_0\times\Omega_1,\Omega'_0\times\Omega'_1$ are $\CC$-relations and $1_\Omega$ is a basic $\CC$-relation, the relation
$\left( (\Omega_0\times\Omega_1)\star (\Omega'_0\times\Omega'_1)\right)\cap 1_\Omega$ is either empty or coincides with $1_\Omega$.
By direct calculations we obtain 
$$
\left( (\Omega_0\times\Omega_1)\star (\Omega'_0\times\Omega'_1)\right)\cap 1_\Omega = 1_{\Omega_0\cap\Omega'_1} \cup 1_{\Omega'_0\cap\Omega_1}.
$$
Therefore $1_{\Omega_0\cap\Omega'_1} \cup 1_{\Omega'_0\cap\Omega_1} = 1_\Omega$, or, equivalently, 
$(\Omega_0\cap\Omega'_1) \cup (\Omega'_0\cap\Omega_1) = \Omega$. This, in turn, implies 
$
\Omega'_1\subseteq\Omega_0, \Omega'_0\subseteq\Omega_1$. Consequently, $
\Omega_0 = \Omega'_1, \Omega'_0 = \Omega_1$ and $\{\Omega_0,\Omega_1\}=\{\Omega'_0,\Omega'_1\}$. A contradiction.
\end{proof}
\begin{prop}\label{040619c}
Let $\XX=(\Omega,\cC)$ be a symmetric coherent J-configuration and $\cA\subseteq M_\Omega(\F)$ its adjacency algebra. 
The  following are equivalent.
\begin{enumerate}
\item[{\rm (a)}] $\XX$ is \textcolor{black}{regular};
\item[{\rm (b)}]  $\XX$ is \textcolor{black}{homogeneous}, i.e. $1_\Omega\in\cC$;
\item[{\rm (c)}]  $\cA\cdot J_\Omega =  \F\langle J_\Omega\rangle$;
\item[{\rm (d)}]  $\cA\star J_\Omega  = \F\langle J_\Omega\rangle$.
\end{enumerate} 
\end{prop}
\begin{proof} The implications (a)$\implies$(b) and (a)$\implies$(c) hold for any rainbow. The implication (c)$\implies$(d) follows directly from the definition of the Jordan product. 

(c)$\implies$(a). It follows from (c) that every matrix $A\in\cA$ has a constant row-sum. In particular, each adjacency matrix $\und{C},C\in\cC$ has this property. Therefore $(\Omega,C)$ is a regular graph for every $C\in\cC$. 

(b)$\implies$(a). Since $\Omega$ is a unique fiber of $\XX$,
every graph $(\Omega,C\cup C^\top),C\in\cC$ is regular by Proposition~\ref{040619b}. By assumption $C=C^\top$. Therefore
$(\Omega,C)$ is a regular graph for every $C\in\cC$, and the configuration is homogeneous. 

(d)$\implies$(c). \textcolor{black}{Pick an arbitrary $A\in\cA$.} A direct calculation shows that $(A\cdot J_\Omega)_{\alpha,\beta} = r(\alpha), (J_\Omega\cdot A)_{\alpha,\beta}=c(\beta)$ where $r(\alpha)$ and $c(\beta)$ stand for $\alpha$-th row and $\beta$-th column sums of $A$. Since $A$ is symmetric, $c(\beta)=r(\beta)$ and $(A\star J_{\Omega})_{\alpha,\beta} = r(\alpha) + r(\beta)$. If $\cA\star J_\Omega = \langle J_\Omega \rangle$ then $r(\alpha)+r(\beta)$ does not depend on a choice of $\alpha$ and $\beta$. Therefore $r(\alpha)$ is constant, i.e. $A$ has a constant row sum. This implies $A\cdot J_\Omega\in\langle J_\Omega\rangle$, as desired.
\end{proof}

\section{Symmetric Jordan configurations of small rank}\label{SJC}

%As before, it is assumed here that $\cJ$ is symmetric, i.e. $\cJ\subseteq \Sm_X(\bR)$. 

The main goal of this section is to prove the following
\begin{theorem}\label{120319a} Let $\XX$ be a symmetric 
coherent J-configuration. If $\cC$ is proper, then $|\cC|\geq 5$. In the case of equality $\XX$ is homogeneous and $\cJ:=\mathbb{R}\langle\und{\cC}\rangle$ is  isomorphic (as a Jordan algebra) to $\bR\oplus\bR \oplus \Sm_2(\bR)$.
\end{theorem}

\textcolor{black}{To prove the theorem we need to remind some properties of CCs of small rank. Recall that any CC of rank at most two is trivial. Thus a non-trivial CC should be of rank three at  least. The properties of those configurations needed in the paper are collected in the subsection below. The standard facts about association schemes which are used here and later on can be found in the monographs~\cite{BI,Z}.
\subsection{Coherent configurations of rank three}\label{r3}
A classical source for the most material presented here is \cite{H75}{,\ Section 12.}}

\textcolor{black}{Let $\XX=(\Omega,\CC)$ be a  CC of rank three, i.e. $|\CC|=3$. According to \cite{H75}
the configuration $\XX$ is homogeneous and commutative, i.e. it is a rank three association scheme. Thus $\CC=\{1_\Omega,C,D\}$ where $1_\Omega,C,D$ form a  partition of $\Omega^2$ .  There are two possibilities: either $C^\top = C,D^\top=D$ or $C^\top = D, D^\top =C$. In the first case all basic relations are symmetric while in the second one all non-reflexive relations are antisymmetric. We refer to those cases as symmetric and anti-symmetric, respectively.}

\textcolor{black}{ In the symmetric case the basic graphs $\Gamma=(\Omega,C)$ and $\overline{\Gamma}=(\Omega,D)$ form a pair of complementary undirected 
 graphs. Moreover, each of these graphs is {\it strongly regular} (an SRG for short), that is it is regular and any pair of graph nodes $\alpha\neq\beta$ has either $\lambda$ or $\mu$ common neighbors depending on whether the nodes are adjacent  
in the graph or not.  The numbers of the quadruple  
$(|\Omega|,k_C,\lambda,\mu)$ are called the {\it  parameters} of the SRG. 
}

\textcolor{black}{
Thus a symmetric association scheme of rank three determines a complementary pair of SRGs. The converse is almost true: any 
complementary pair of SRGs, besides the one consisting of complete and empty graphs, yields a symmetric scheme of rank three.
The following characterization of SRGs (\cite{IS}{, Lemma 7.2.9}) will be needed in what follows.
\begin{prop}\label{160920a} An undirected regular graph is strongly regular if and only if it has at most three eigenvalues.
\end{prop}
%Since $\Q\langle I_\Omega,\und{C},\und{D}\rangle$ is a matrix algebra, the adjacency matrix 
%$\und{C}$ satisfies the equation $\und{C}^2 =k_C I_\Omega + \lambda \und{C} + \mu\und{D}$ for suitable 
%scalars $\lambda,\mu\in\Q$. These scalars have a combinatorial meaning: $k_C$ is the valency the graph: $\lambda$ is the number of common neighbors of two adjacent vertices; $
%\mu$ is the number of common neighbors of two non-adjacent vertices. 
%numbgraphs are {\it strongly regular} \cite{??}
}

\textcolor{black}{In the antisymmetric case both basic graphs 
$\Gamma=(\Omega,C)$ and $\overline{\Gamma}=(\Omega,D=C^\top)$ are orientations of a complete graph, i.e. they form a complementary pair of tournaments. Such tournaments have an additional property: they are {\it doubly regular}, \cite{RB}. In this case $k_C=k_D$ is an odd number \cite{H75} and $|\Omega|=2k_C+1$ is equal $3$ modulo $4$. }
\textcolor{black}{\subsection{Proof of Theorem~\ref{120319a}}}

\textcolor{black}{We split the proof into a sequence of separate statements. The first one deals with the homogeneous case.} 
\begin{prop}\label{050619a} Assume that \textcolor{black}{$\XX=(\Omega,\CC)$ is a homogeneous symmetric rainbow of rank at most $4$.} Then  
$\XX$ is a Jordan scheme if and only if $\XX$ is a commutative association scheme. 
\end{prop}
\begin{proof} Denote $\cJ:=\bR\langle\und{\CC}\rangle$.
If $\XX$ is an association scheme, then $\cJ$ is commutative \textcolor{black}{because it consists of symmetric matrices.} 
%(see \cite{H75}{(4.1)})
Therefore, the $\star$-product coincides with the usual one. In this case $\XX$ is a Jordan scheme.  

To prove the converse implication let us write $\cC=\{C_1=1_\Omega,..., C_r\}, r\leq 4$ and denote $A_i:=\und{C_i}$. 
Then  $\cJ = \langle A_1,...,A_r\rangle$ \textcolor{black}{is a Jordan algebra}. It follows from 
$A_i^k = A_i^{\star k} \in\cJ,k\in\Z_{\geq 0}$ that the minimal polynomial of each $A_i$ has degree at most $r$. If this bound is reached for some $i$, then $\cJ=\bR[A_i]$ implying that the matrices $A_1,...,A_r$ pairwise commute. Therefore $A_i\star A_j = A_i\cdot A_j$ and  $\XX$ is a commutative \as.  

If the minimal polynomial of every $A_i$ has degree strictly less than $r \leq 4$, then each basic graph $\Gamma_i = (\Omega, C_i)$ has at most $3$ eigenvalues implying that every $\Gamma_i$ is a strongly regular graph \textcolor{black}{ by Proposition~\ref{160920a}. If $r=1,2$ then $\XX$ is a trivial scheme. If $r=3$, then $\Gamma_2$ and 
$\Gamma_3$ form a complementary pair of SRGs, and, therefore $\CC=\{1_\Omega,C_2,C_3\}$ is a rank three association scheme (see subsection~\ref{r3}).} If $r=4$, then we obtain a partition of a complete graph $K_\Omega$ into a disjoint union of three strongly regular graphs. According to \cite{VDM}{, Theorem 2}, $\XX$ is an association scheme. 
\end{proof}

The statement below shows that a non-homogeneous case cannot appear in Theorem~\ref{120319a}.
\begin{prop}\label{090319c} \textcolor{black}{If $\XX=(\Omega,\CC)$ is a non-homogeneous symmetric Jordan configuration of rank at most five}, then $\XX$ is a symmetrization of a direct sum of two  homogeneous CCs\footnote{\textcolor{black}{For the definition of a direct sum of CCs we refer the reader to \cite{CP}{,\ Section 3.2.1}.}}.
\end{prop}
\begin{proof} According to Proposition~\ref{040619b} there exists 
a fiber decomposition $\Omega = \Omega_1\cup ...\cup \Omega_f$. 
Then $\cC = \bigcup_{1\leq a\leq b\leq f} \cC^{ab}$ implying that $|\cC|=
\sum_{1\leq a\leq b\leq f} |\cC^{ab}|$. 
Together with 
$|\cC^{aa}|\geq \min(2,|\Omega_a|)$ 
and $|\cC^{ab}|\geq 1$ we conclude that $|\cC| \geq 6$ if $f\geq 3$.
Now the assumption $|\cC|\leq 5$ yields us $f=1,2$. Since $\XX$ is non-homogeneous, $f=2$.

Every basic graph $(C,\Omega),C\in\CC^{12}$ is undirected, bi-partite and bi-regular (Proposition~\ref{040619b}, part (b)). \textcolor{black}{Therefore $C(\omega)\neq\emptyset$ for any $\omega\in\Omega$. Pick an arbitrary $\omega\in\Omega_i$. Then $\{C(\omega)\}_{C\in \CC^{12}}$ is a partition of $\Omega_{3-i}$. Therefore $|\CC^{12}|\leq|\Omega_{3-i}|$ implying} $|\CC^{12}|\leq \min(|\Omega_1|,|\Omega_2|)$. 

If $|\cC^{12}|\geq 2$, then $|\Omega_i|\geq 2,i=1,2$, and, therefore, $|\cC^{11}|,|\cC^{22}|\geq 2$ contrary to the assumption $|\cC|\leq 5$. Hence, $|\cC^{12}|=1$ and $|\cC^{11}|+|\cC^{22}|\leq 4$.

Since $\XX_i:=(\Omega_i,\cC^{ii}),i=1,2$ is a homogeneous symmetric Jordan configuration of rank at most $4$, it is a symmetric association scheme by Proposition~\ref{050619a}. Combining this with $|\cC^{12}|=1$ we conclude that $\cC$ is a symmetrization of the rainbow
$\cC^{11}\cup \cC^{22}\cup\{\Omega_1\times\Omega_2,\Omega_2\times\Omega_1\}$. The latter partition is a direct sum of the association schemes $\XX_1$ and $\XX_2$.  
\end{proof}
\noindent{\bf Remark.} A more delicate analysis shows that if $|\Omega_i| > 1,i=1,2$, then $\XX_1$ and $\XX_2$ are trivial schemes
and $|\cC|=5$. 

Now we are ready to prove the main result of the subsection.

\noindent {\bf Proof of Theorem~\ref{120319a}.}\\
 Assume that $|\cC|\leq 5$. By Proposition~\ref{090319c} if $\XX$ is non-homogeneous, then it is non-proper. Therefore $\XX$ is homogeneous. If $|\cC|\leq 4$, then by Proposition~\ref{050619a} 
$\XX$ is a symmetric association scheme contrary to being a proper Jordan scheme. Thus we may assume that $|\cC|=5$. Set $\cJ:={\mathbb R}\langle \und{\CC}\rangle$.

If $(\cJ,\star)$ is associative, then Proposition~\ref{120619a} (see the Appendix) implies that $\XX$ is a symmetric association scheme. Thus in this case $\XX$ is an improper JC. So, we may assume that $(\cJ,\star)$ is not associative. 

 Since $\XX$ is homogeneous, the idempotent $E_0:=|\Omega|^{-1} J_\Omega$ satisfies $E_0\star \cJ = \langle E_0\rangle$ (Proposition~\ref{040619c}). Thus $\cJ$ has a direct sum decomposition $\cJ = \cJ\star E_0 + \cJ \star (I_\Omega - E_0)$. The Jordan subalgebra $\cJ_1:=\cJ \star (I_\Omega - E_0)$ has dimension $4$. Its unit coincides with $E_1: = I_\Omega - E_0$.

Since $\cJ$ is {\it formally real} (\cite{J}), its subalgebra  $\cJ_1$ is formally real too, and, by  Jordan-von Neumann-Wigner Theorem, $\cJ_1$ is a direct sum of simple Jordan algebras from the following list (see \cite{J,M}):
\begin{enumerate}
\item[1)] $\bR$;
\item[2)]the Jordan algebra $\bR\oplus_f V$ of a positive definite symmetric bilinear form $f$ on a real vector space $V,\dim(V)\geq 2$;
\textcolor{black}{
\item[3)] $H_n(\mathbb{D}):=\{A\in M_n(\mathbb{D})\,|\, \overline{A}^\top = A\}, n\geq 2$, where  $\mathbb{D}\cong\bR, \mathbb{C},\mathbb{H}$ and \  $\bar{ {} }$ \ is the standard conjugation in $\mathbb{D}$ (here and later on $\mathbb{H}$ stands for the algebra of quaternions);
\item[4)] $H_3(\mathbb{O}):=\{A\in M_3(\mathbb{O})\,|\, \overline{A}^\top = A\}$ where $\mathbb{O}$ is the algebra of octonians and  \  $\bar{ {} }$ \  is the standard conjugation in $\mathbb{O}$. 
}
\end{enumerate}
If $\cJ_1=\bR^4$, then both $\cJ_1$ and $\cJ$ are associative implying that $\cJ$ is associative, contrary to the assumption.  

Assume now that at least one of the simple summands of $\cJ_1$ is non-associative. Then it has dimension at most $4$. The only algebras in the above list
satisfying this dimension restriction are either of type (2) with $\dim(V)=2,3$ or of types (3)-(4) satisfying 
$4\geq \dim(H_n(\mathbb{D}))=n + \frac{n(n-1)}{2}\dim(\mathbb{D})$ and $n\geq 2$. In the latter case $n=2$ and $\mathbb{D}\cong\mathbb{R},\mathbb{C}$. Thus  $\cJ_1$ is one of the following 
\begin{enumerate}
\item[1)] $\bR\oplus (\bR\oplus_f V), \dim(V)=2$;
\item[2)] $\bR\oplus_f V, \dim(V)=3$;
\item[3)] $\bR\oplus \sym_2(\bR)$; 
\item[4)] $H_2(\mathbb{C})$. 
\end{enumerate} 
The first and the third algebras are isomorphic and in these cases $\cJ\cong_J \bR\oplus\bR\oplus \Sm_2(\bR)$ hereby providing the conclusion of the Theorem. 

It remains to deny the second and the fourth cases. In these cases every element $x\in \cJ_1$ is {\it quadratic}, meaning that $x^{\star 2}$ is a linear combination of $x$ and the identity $E_1$. This implies that any $x\in \cJ$ has minimal polynomial of degree at most three. Therefore, any union of non-identical basic relations of $\cC=\{C_0=I_\Omega, C_1,C_2,C_3,C_4\}$ is a strongly regular graph. This implies that for any permutation $i_1,i_2,i_3,i_4$ of the indices $\{1,2,3,4\}$ the relations $C_{i_1},C_{i_2},C_{i_3}\cup C_{i_4}$ form a partition of a complete graph into a disjoint union of three SRGs. \textcolor{black}{By \cite{VDM}{, Theorem 2}} these relations together with $C_0$ form a symmetric \as\ implying that 
$\und{C}_{i_1},\und{C}_{i_2}$ commute. Thus we have shown that any two basic matrices of $\cJ$ commute implying that $(\cJ,\star)$ is associative. A contradiction. \hfill $\square$
%\end{proof}

\medskip

\noindent{\bf Remark.} Note that the algebra $H_3(\mathbb{O})$ never appears in a decomposition of a symmetric coherent J-algebra $\cJ$, because it is an exceptional Jordan algebra while any subalgebra of $(M_\Omega(\mathbb{R}),\star)$ is special. We do not know which one of the special algebras appearing in Jordan-von Neumann-Wigner Theorem could appear in the decomposition of a  symmetric coherent J-algebra.

\begin{theorem}\label{100319a} \textcolor{black}{Let $\cJ\subseteq \sym_\Omega({\mathbb R})$ be a $\star$-subalgebra isomorphic to
$\cJ\cong\bR\oplus\bR\oplus \Sm_2(\bR)$. Then the $\cdot$-subalgebra of $M_\Omega({\mathbb R})$  generated by $\cJ$ is isomorphic to $\bR\oplus\bR\oplus M_2(\bR)$.}
\end{theorem}
\begin{proof} 
%As before, $A_0,A_1,...,A_4$ stands or a standard basis of $\cJ$. 
It follows from the assumption that $[\cJ,\cJ]\neq 0$ (otherwise, $\cJ\cong_J\bR^5$ contrary to $\cJ\cong\bR\oplus\bR\oplus \Sm_2(\bR)$).

Denote by $\varphi$ an isomorphism from  $\bR\oplus\bR\oplus \Sm_2(\bR)$ onto $\cJ$ and denote $E_0=\varphi(e_0), E_1=\varphi(e_1), E_2=\varphi(e_2)$ where $e_0=(1,0,O_2),e_1=(0,1,O_2),e_2=(0,0,I_2)$.
It follows from $e_i^{\star 2} = e_i$ that $E_i^2 = E_i^{*2}=E_i$. i.e.  $E_i$ are idempotents. It follows from $e_i\star x = e_i\star (e_i\star x)$ that $E_i A + A E_i = 2 E_i A E_i$. Therefore $E_i A = E_i A E_i = A E_i$ holds for all $A\in\cJ$. In other words $E_i$ $\cdot$-commutes with any element $A\in\cJ$. Combining this with $i\neq j\implies e_i\star e_j = 0$ we conclude that $i\neq j \implies E_i E_j = O$. Thus $E_0,E_1,E_2$ are pairwise orthogonal central idempotents of the $\cdot$-subalgebra generated by $\cJ$. 

Pick an arbitrary pair $A,B\in\cJ$  such that the tuple $E_0,E_1,E_2,A,B$ forms a basis of $\cJ$.
If $A$ and $B$ commute, then $[\cJ,\cJ]=0$. A contradiction.
Therefore $[A,B]\neq 0$.

Consider the vector space $\cA=\cJ+\langle AB\rangle$. \textcolor{black}{We are going to prove that $\cA$ is a $\cdot$-subalgebra of the full matrix algebra.} Since $A,B$ are symmetric 
and don't commute, their product $AB$ is not a symmetric matrix. Therefore $AB\not\in\cJ$ and $\dim(\cA)=6$.
It follows from $AB+BA\in\cJ$ that $\cA = \cJ +\langle BA\rangle$. 
\textcolor{black}{Thus $\cJ +\langle AB\rangle = \cA = \cJ +\langle BA\rangle$ implying $\cA^\top = \cA$.}

It follows from $AE_i = A\star E_i\in\cJ,A^2\in\cJ$ and $ABA\in\cJ$ (see (2) in \cite{S}) that \\
$A\cA = A\langle E_0,E_1,E_2,A,B,BA\rangle\subseteq \cJ + \langle AB\rangle=\cA$. Analogously,
$B\cA \subseteq \cA$.  \textcolor{black}{Applying ${}^\top$ to the  inclusions $A\cA\subseteq \cA\supseteq B\cA$ we conclude that $\cA A\subseteq \cA\supseteq \cA B$.}

The inclusion $E_i\cJ\subseteq \cJ$ implies that $E_i\cA \subseteq \cJ +\langle E_i AB\rangle\subseteq \cJ + \cJ B\subseteq \cJ + \langle AB \rangle = \cA$. Thus, we have proven that $\cA$ is a closed with respect to the usual matrix product.

The algebra $\cA$ is \textcolor{black}{${}^\top$-closed, and, therefore, is} a semisimple non-commutative $6$-dimensional algebra with central idempotents $E_0,E_1,E_2$. Since both $E_0\cA$ and  $E_1\cA$ are one-dimensional ideals, the ideal $E_2\cA$ is isomorphic either to $\mathbb{H}$ or $M_2(\bR)$.  To finish the proof we have to eliminate the possibility of $E_2\cA\cong\mathbb{H}$. Assume that this is the case. Then the intersection $\cJ\cap E_2\cA$ has dimension at least three. Since every matrix of $\cJ$ has real eigenvalues, the intersection $\cJ\cap E_2\cA$ is a three-dimensional subspace of $\mathbb{H}$ the elements of which have minimal polynomials with real roots. This is a contradiction, because the minimal polynomial of every non-zero imaginary quaternion has a form $x^2+a, a>0$. 
\end{proof}

We finish this section by a useful sufficient condition for being a Jordan scheme.

\begin{prop}\label{p3} Let $\XX=(\Omega,\cC=\{C_0=1_\Omega,C_1,C_2,C_3,C_4\})$ be a symmetric homogeneous rainbow of $\Omega^2$. Denote $C:=C_2\cup C_3 \cup C_4, A:=\und{C}$. Assume that for any $i\in\{2,3,4\}$ the partition $\cC_i=\{C_0,C_1,C_i,C\setminus C_i\}$ is an \as. Then $\XX$ is  a \js.
\end{prop}
\begin{proof}
Denote by $A_i$ the adjacency matrix of $C_i$ and by $\cA$ the linear span of $A_0,...,A_4$. 
By assumption the vector space $\cA_i:=\langle A_0, A_1, A_i, A - A_i\rangle, i\in\{2,3,4\}$ is the adjacency algebra of the \as\ $\cC_i$.

Since $A_i\star A_i = A_i^2\in\cA_i\subseteq \cA$, one has to prove that the inclusion $A_j\star A_k\in\cA$ holds if $k\neq j$. It follows from $2A_k\star A_j = (A_j+A_k)^2 - A_k^2 - A_j^2$ that it is sufficient to show that $(A_k + A_j)^2\in\cA$ holds for all $k\neq j\in\{1,2,3,4\}$.

\medskip

\noindent {\bf Case 1.} One of $k,j$ equals $1$. W.l.o.g. we may assume that $k=1$. Now
 $j\neq 1$ and then $(A_1 + A_j)^2\in\cA_j\subseteq\cA$. 

\medskip

\noindent {\bf Case 2.} Both $k$ and $j$ are distinct from $1$.

Note that in this case $k,j\in\{2,3,4\}$.
It follows from $k\neq j$ that $A_k+A_j\in\cA_\ell$ where $\ell$ is the unique element of $\{2,3,4\}\setminus\{k,j\}$. Hence $(A_k+A_j)^2\in\cA_\ell\subseteq \cA$.
\end{proof}
{\bf Remark.} It follows from the assumptions of the above proposition that \textcolor{black}{$(\Omega, \{1_\Omega,C_1,C\})$ is a symmetric CC of rank three. Hence $(\Omega,C_1)$} is a strongly regular graph.

\section{A prolific construction of rank five Jordan schemes based on the WFDF-construction}\label{WFDF}

In this section we provide an infinite series of examples based on a prolific construction of SRGs proposed by Wallis and Fon-Der-Flaass \cite{W},\cite{FDF} ({\it WFDF-construction} in brief). The construction is called prolific because it produces hyperexponentially many non-isomorphic SRGs sharing the same parameter set. Note that our presentation of this construction is a bit different from \cite{W,FDF}.

We start with an affine design with point set $V=\Z_3^d$.
The blocks of the design are the affine hyperplanes of the vector space 
$V$. There are exactly $r:=\frac{3^d-1}{2}$ hyperplanes going through
zero. We denote them as $H_1,...,H_r$ where labelling is arbitrary.

For each hyperplane $H_i$ we pick 
an arbitrary linear epimorphism $\pi_i:V\rightarrow\Z_3$ with $\ker(\pi_i)=H_i$ (there are two choices of $\pi_i$ for each $i$).
  
We are going to build a rank five \js\  on the set $\Omega:=V\times\{0,1,...,r\}$. In the provided construction each color class will be a strongly regular graph. 

For the rest of the section we use the following abbreviations:
 $[0,r]:=\{0,1,...,r\}, \Omega_i = V\times\{i\}, v_i:=(v,i)$. 

The first relation of our scheme, called $S$, \textcolor{black}{comes from the equivalence relation on $\Omega$ corresponding to the partition $\Omega = \Omega_0\cup ... \cup \Omega_r$. It is defined  as follows}
$S=\{(u_i,v_j)\,|\, i=j \land u\neq v\}$. The basic graph 
corresponding to $S$ is a disjoint union of $r+1$ copies of a complete graph $K_{3^d}$. 

Three other basic graphs will be strongly regular with parameters\\ $\left(3^d\frac{3^d+1}{2},3^{d-1}\frac{3^d-1}{2},3^{d-1}\frac{3^{d-1}-1}{2},3^{d-1}\frac{3^{d-1}-1}{2}\right)$. 
Altogether we obtain a partition of the complete graph $K_{\Omega}$ into a disjoint union of four strongly regular graphs one of which is disconnected. To build the connected basic graphs we define on the set $[0,r]$ an arbitrary binary operation $\diamond$ subject to two conditions 
\begin{equation}\label{eq1}
\begin{array}{l}
\forall_{a\in [0,r]} \ \ a\diamond a = 0;\\
\forall_{a\in [0,r]}\mbox{ the mapping }x\mapsto a\diamond x,x\in [0,r]\mbox{ is a bijection}.
\end{array}
\end{equation}
One can easily count that the number of such operations is $(r!)^{r+1}$. One of the choices is $x\diamond y = x - y$ where the subtraction is done modulo $r+1$.

Now, for every ordered pair $(i,j), 0\leq i < j \leq r$ pick an arbitrary permutation $\sigma_{ij}\in\sym(\Z_3)$. For an ordered pair $(i,j)$ with $i > j$ we set $\sigma_{ij}:=\sigma_{ji}^{-1}$.  

Given a binary operation $\diamond$ satisfying \eqref{eq1} and an  array of bijections\\ $\Sigma:=\left( \sigma_{ij} \right)_{0\leq i < j\leq r},\sigma_{ij}\in\sym(\Z_3)$, we define a binary relation $R:=R(\diamond,\Sigma)$ on $\Omega$ as follows 
\begin{equation}\label{eq22}
(u_i,v_j)\in R \iff i\neq j\ \land\ \sigma_{ij}(\pi_{i\diamond j}(u)) =  \pi_{j\diamond i}(v).
\end{equation}
The relation $R$ is a symmetric binary relation on the set $\Omega$.
It determines an undirected graph $\Gamma:=(\Omega,R)$. It follows directly from the construction that each $\Omega_i$ is a coclique of $\Gamma$. 

The statement below is a key one in this Section. Although parts of it may be retrieved from the papers~\cite{W,FDF}, we provide here  complete proofs to make the text self-contained.
\begin{prop}\label{p4} The following properties hold.
\begin{enumerate}
\item[(a)] The graph $\Gamma$ is a strongly regular graph with parameters  $$\left(3^d\frac{3^d+1}{2},3^{d-1}\frac{3^d-1}{2},3^{d-1}\frac{3^{d-1}-1}{2},3^{d-1}\frac{3^{d-1}-1}{2}\right).$$
\item[(b)] 
The partition $\Omega =\Omega_0\cup ... \cup \Omega_r$ is a Hoffman's coloring.\footnote{A Hoffman coloring is a proper vertex coloring in which every color class meets Hoffman's coclique bound, see~\cite{HT} for details.}
\item[(c)] The partition $1_\Omega, S, R, \Omega^2\setminus(1_\Omega \cup S\cup R)$  
is a symmetric three-class imprimitive \as.
\end{enumerate}
\end{prop}
\begin{proof} 

{\bf Part (a).} First we note that the number of points is $|V|(r+1) = 3^d\frac{3^d+1}{2}$.

Let us prove now that $\Gamma$ is regular. 
Pick an arbitrary point $u_i\in\Omega_i$. According to~\eqref{eq22} a point $x_j$ is connected to $u_i$ (both $u$ and $x$ belong to $V$) iff $j\neq i$ and $\sigma_{ij}(\pi_{i\diamond j}(u)) = \pi_{j\diamond i}(x)$. Since $\pi_{j\diamond i}$ is a linear function, the  solutions of above equation form an affine hyperplane in $V$. \textcolor{black}{ Therefore, for a fixed $j$ the number of solutions is $3^{d-1}$. Hence $u_i$ is adjacent to $3^{d-1}$ points of $\Omega_j$. This yields us $3^{d-1}\cdot r  = 3^{d-1}\frac{3^d-1}{2}$ points adjacent to $u_i$.}

Now we show that any pair $u_i,v_j$ of distinct points have the same number of common neighbors, namely $3^{d-1}\frac{3^{d-1}-1}{2}$.

Assume first that $i=j$. Then $x_k$ is connected to both $u_i$ and $v_i$ iff $k\neq i$ and 
$$\left\{
\begin{array}{rcl}
\sigma_{ik}(\pi_{i\diamond k}(u)) & = & \pi_{k\diamond i}(x);\\
\sigma_{ik}(\pi_{i\diamond k}(v)) & = & \pi_{k\diamond i}(x).
\end{array}\right.
$$
It is easy to see that the system is consistent iff 
$$\pi_{i\diamond k}(u) = \pi_{i\diamond k}(v)\iff u-v\in\ker(\pi_{i\diamond k}) = H_{i\diamond k}.$$
In the latter case the above system has $3^{d-1}$ solutions.
Since every non-zero vector of $V$ is contained in $\frac{3^{d-1}-1}{2}$ hyperplanes, there exist $\frac{3^{d-1}-1}{2}$ indices $k\in [0,r]\setminus \{i\}$ with $u-v\in H_{i\diamond k}$. Therefore the number of joint neighbors is $3^{d-1}\frac{3^{d-1}-1}{2}$ in the case of $i=j$.

Assume now that $i\neq j$. Then $x_k$ is connected to both $u_i$ and $v_j$ iff $k\neq i,j$ and 
$$\left\{
\begin{array}{rcl}
\sigma_{ik}(\pi_{i\diamond k}(u)) & = & \pi_{k\diamond i}(x);\\
\sigma_{jk}(\pi_{j\diamond k}(v)) & = & \pi_{k\diamond j}(x).
\end{array}\right.
$$
The above system is always consistent, since $i\neq j\implies k\diamond i\neq k\diamond j\implies \pi_{k\diamond i}$ and $\pi_{k\diamond j}$ are linearly independent.  The number of solutions for a fixed $k$ equals to $|H_{k\diamond i}\cap H_{k\diamond j}|=3^{d-2}$.
Multiplying by the number of $k$'s distinct from $i,j$ we conclude that 
$u_i,v_j$ have $3^{d-2}(r-1) = 3^{d-2}\frac{3^d - 3}{2} = 
3^{d-1}(\frac{3^{d-1} - 1}{2})$ common neighbors, as claimed.  

\medskip

\noindent{\bf Proof of part (b).} It follows from the construction of $\Gamma$ that each $\Omega_i$ is a coclique of $\Gamma$. Therefore the partition $\Omega = \Omega_0\cup ...\cup \Omega_r$ is a coloring of $\Gamma$. The sizes of $\Omega_i$ are equal to $3^d$ which meets Hoffman's coclique bound.

\medskip

\noindent{\bf Proof of part (c).}
It was shown in part (b) that
$\Omega=\Omega_0\cup ...\cup \Omega_r$ is a Hoffman coloring  of $\Gamma$. By Proposition 4.1, \cite{HT} the relations  $1_\Omega, S, R, \Omega^2\setminus(1_\Omega \cup S\cup R)$  form an association scheme. 
\end{proof}
For each pair $i < j$ of indices we choose an arbitrary $3$-cycle $\theta_{ij}\in\{(0,1,2),(0,2,1)\}$. Write $\Theta = (\theta_{ij})_{0\leq i < j\leq r}$, $\Theta \Sigma= (\theta_{ij}\sigma_{ij})_{0\leq i < j\leq r},\Theta^2\Sigma =(\theta_{ij}^2\sigma_{ij})_{0\leq i < j\leq r}$ (product here is the product of permutations in $\sym(\Z_3)$). 

\begin{theorem}\label{t1}
Define  $R_1:=R(\diamond,\Sigma),R_2:=R(\diamond,\Theta\Sigma),R_3:=R(\diamond,\Theta^2\Sigma)$. Then the relations $1_\Omega,S,R_1,R_2,R_3$ form a \js.
\end{theorem}
\begin{proof} 

First we show that $I_\Omega, S,R_1,R_2,R_3$ form a symmetric homogeneous rainbow.

It follows from Proposition~\ref{p4} that $R_1,R_2,R_3$ and $S$ are symmetric and regular relations. It follows from the construction that $S$ intersects trivially each of the relations $R_1,R_2,R_3$.

If $(u_i,v_j)\in R_a\cap R_b$ with $a\neq b$, then
$i\neq j$ and we may assume that $i < j$ (because, both $R_a$ and $R_b$ are symmetric). It follows from~\eqref{eq22} that
$$
\left.
\begin{array}{rcl}
\theta_{ij}^{a-1}(\sigma_{ij}(\pi_{i\diamond j}(u))) & = & \pi_{j\diamond i}(v),\\
\theta_{ij}^{b-1}(\sigma_{ij}(\pi_{i\diamond j}(u))) & = & \pi_{j\diamond i}(v)
\end{array}\right\}\implies
\theta_{ij}^{b-a}(\sigma_{ij}(\pi_{i\diamond j}(u))) = \sigma_{ij}(\pi_{i\diamond j}(u)),
$$
a contradiction, since $\theta_{ij}^{b-a}$ is a 3-cycle on $\Z_3$ and has no fixed points. Thus we have proven that $S,R_1,R_2,R_3$ are pairwise disjoint. Let us show now that their union coincides with a complete graph. Pick an arc $e=(u_i,v_j)$ of the complete graph $K_\Omega$. If $i = j$, then $e\in S$. If $i\neq j$, then, replacing $e$ by $(v_j,u_i)$, if necessary, we may assume that $i < j$ . Since $\theta_{ij}$ is a 3-cycle on $\Z_3$ and both $\sigma_{ij}(\pi_{i\diamond j}(u))$ and $\pi_{j\diamond i}(v)$ are elements of $\Z_3$, there exists a power of $\theta_{ij}$ which moves the first element into the second one, i.e.
$$
\theta_{ij}^{a-1}(\sigma_{ij}(\pi_{i\diamond j}(u))) = \pi_{j\diamond i}(v))
$$
for some $a\in\{1,2,3\}$. This implies $(u_i,v_j)\in R(\diamond,\Theta^{a-1}\Sigma) = R_a$.

Thus we have shown that $S,R_1,R_2,R_3$ form a symmetric regular partition of the complete graph. It follows from Proposition~\ref{p4}, part (c) that this partition satisfies the assumptions of Proposition~\ref{p3}. Therefore it is a \js.
\end{proof}
The statement below provides sufficient conditions when the Jordan schemes constructed in Theorem~\ref{t1} are proper. 
\begin{prop}\label{160619a} Let $\XX = (\Omega,\{1_\Omega,S,R_1,R_2,R_3\})$ be an arbitrary rank five symmetric Jordan scheme of order 
$3^d\frac{3^d + 1}{2}$ and valencies $1, 3^d-1, 3^{d-1}\frac{3^d-1}{2}, 3^{d-1}\frac{3^{d}-1}{2},3^{d-1}\frac{3^{d}-1}{2}$ where $d$ is an even integer. Assume that the basic graph $(\Omega,S)$ is a disjoint union of complete graphs. Then the scheme is proper.
\end{prop}
\begin{proof} Assume, towards a contradiction, that $\XX'=(\Omega,\cC')$ is a CC such that its symmetrization $\tilde{\XX'}$ is a Jordan scheme with the parameters described above. First, we note that $\XX'$ is homogeneous, since otherwise at least one of the symmetrized relations is a bipartite graph of even order while the scheme order is odd. It follows from the assumptions  that $\XX'$ is imprimitive with blocks of size $3^d$ formed by the cliques of $(\Omega,S)$. We denote these cliques by $\Omega_i$.

If $S = T\cup T^\top$, then both $T$ and $T^\top$ are anti-symmetric regular relations of degree $\frac{3^d-1}{2}$. Now one can realize that the restriction
$(\Omega_1,\{I_{\Omega_1},T\cap\Omega_1^2,T^\top\cap\Omega_1^2\})$ is a non-symmetric scheme of rank three. In this case $|\Omega_1|$ should be equal to $3$ modulo $4$  \textcolor{black}{(see the Subcestion~\ref{r3}).} But this contradicts to $|\Omega_1|=3^d\equiv\,1({\rm mod}\ 4)$.
Thus $S$ is a symmetric basic relation of $\XX'$.

Assume now that $R_i$ is not a basic relation of $\XX'$ for some $i$. Then $R_i = T_i\cup T_i^\top$ where $T_i$ is a suitable anti-symmetric basic relation of $\XX'$.
The valency of $T_i$ is $k/2$ where $k$ stands for the valency of $R_i$ (recall that $k=3^{d-1}\frac{3^d-1}{2}$).
The product $\und{T}_i\cdot\und{T}_i^\top$ is a symmetric matrix. Therefore 
it is a linear combination of $I_\Omega,\und{S},\und{R}_1,\und{R}_2,\und{R}_3$ with non-negative integers: 
$$
\und{T}_i\cdot\und{T}_i^\top = \frac{k}{2} I_\Omega + a \und{S} + b_1
\und{R}_1+b_2\und{R}_2+b_3\und{R}_3.
$$
This equality implies that $(k/2)^2 - (k/2)$ is divisible by the greatest common divisor $g$ of the valencies
$k_S= 3^d - 1,k_{R_1}=k_{R_2}=k_{R_3}=k = 3^{d-1}\frac{3^d-1}{2}$. 
A simple calculation yields us $g=\frac{3^d - 1}{2}$. Thus 
$(k/2)^2 - (k/2)= 3^{d-1}\frac{3^d-1}{4}(3^{d-1}\frac{3^d-1}{4} -1)$ is divisible by 
$\frac{3^d - 1}{2}$. Therefore, the factor $(3^{d-1}\frac{3^d-1}{4} -1)$ should be even. On the other hand, $3^d - 1$ is divisible by eight, because $d$ is even. This implies that $\frac{3^d-1}{4}$ is even too. But in this case $3^{d-1}\frac{3^d-1}{4} -1$ is odd. A contradiction. 

Thus we can conclude that $\XX'=\XX$, that is $\XX$ is a symmetric association scheme with $4$ classes. It remains to show that such a scheme does not exist. So, assume, towards a contradiction, that it exists. Then it is a commutative and imprimitive scheme with a closed subset $\mathcal{E}:=\{1,S\}$. If one of the non-trivial algebraic $\mathcal{E}$-cosets contains one element, say $R_1$, then $k_{R_1}=3^{d-1}\frac{3^d-1}{2}$ is divisible by $k_{\mathcal{E}}=3^d$ (Lemma 2.3.4 (i),\cite{Z}), a contradiction. Therefore all three relations $R_1,R_2,R_3$ belong to one coset. This implies that $\und{R}_i (\und{S} + I_\Omega) = 
\lambda_i (\und{R}_1+\und{R}_2+\und{R}_3)$ (Lemma 2.3.1, \cite{Z}). Comparing the valencies in both sides we obtain $3^{d-1}\frac{3^d-1}{2} 3^d = \lambda_i 3^d \frac{3^d-1}{2}\implies\lambda_i = 3^{d-1}$.
Now we obtain 
$$\und{R}_i (\und{S} + I_\Omega) = 
3^{d-1} (\und{R}_1+\und{R}_2+\und{R}_3) \implies \und{R}_1 \und{S}  = (3^{d-1}-1) \und{R}_1 + 3^{d-1} (\und{R}_2+\und{R}_3).
$$
By the so called triangle property of the association scheme \textcolor{black}{(\cite{BI},Proposition 2.2 (vi))} we obtain
$$
p_{R_1,S}^{R_2} k_{R_2} = p_{R_2,R_1}^S k_S\iff
3^{d-1} \cdot 3^{d-1} \frac{3^d-1}{2} = p_{R_2,R_1}^S (3^d-1)
\implies p_{R_2,R_1}^S = \frac{3^{2d-2}}{2}\not\in\Z.
$$
A contradiction.
\end{proof}
{\bf Remark.} We think that in the case of $d>1$ being odd most of the Jordan schemes \textcolor{black}{constructed in this section} are also proper.  We still did not find a proof for that. The only thing we can show  is that if the corresponding  Jordan scheme is non-proper, then it is a fusion of a rank $6$ non-commutative scheme. This is because the valencies of $R_1,R_2,R_3$ are odd, and, for this reason, none of those relations can split into a union of an anti-symmetric relation and its transposed. The relation $S$ in this case splits into a pair $U,U^\top$ and $\{1_\Omega,U,U^\top\}$ is an anti-symmetric normal closed subset. We refer a reader to~\cite{HZ,FZ} where \textcolor{black}{such association schemes} are studied in detail. 

Note that if $d=1$, then the construction yields a unique Jordan scheme. This scheme is non-proper and coincides with the symmetrization of the thin scheme of the group $S_3$\textcolor{black}{, cf. \cite{S59}, as well as \cite{KMR}, Section 21.}

\section{Jordan schemes constructed by switching in non-commutative association schemes}\label{switching}

\newcommand{\fT}{{\mathfrak T}}

Let $\fT:=(\Omega,\cR=\{C_0,C_1,....,C_{m-1},S_0,S_1,...,S_{m-1}\})$ be an \as\ of order $m(n+1)$  with the following multiplication table ($C_0$ is the identity relation), cf. \cite{KS,SR}:
\begin{equation}\label{eq:310319a}
\begin{array}{c}
\und{C}_i \cdot \und{C}_j = \und{C}_{i+j}, \und{C}_i\cdot \und{S}_j = \und{S}_{i+j},\und{S}_j\cdot \und{C}_i = \und{S}_{j-i},\\
\und{S}_i\cdot \und{S}_j = n \und{C}_{i-j} + \frac{n-1}{m}(\und{S}_0+...+\und{S}_{m-1}),
\end{array}
\end{equation}
here the arithmetic is done modulo $m$ and $m\,|\, (n-1)$.

The existence of such schemes was shown in \cite{KS,SR,KRW}. The valencies of $C_i$'s are one, while the valencies of $S_i$'s are $n$. The graphs $(\Omega,S_i),i=0,...,m-1$ are pairwise isomorphic distance regular antipodal covers of $K_{n+1}$.

The scheme has the unique non-trivial closed subset $\cC:=\{C_0,C_1,...,C_{m-1}\}$.
The union $E:=C_0\cup C_1\cup ...\cup C_{m-1}$ is an equivalence relation with $n+1$ classes (called fibers in what follows) of size $m$. Note that 
$\cC$ is the thin radical in the sense of Zieschang \cite{Z}.

The symmetrization of the above scheme yields us a \textcolor{black}{non-proper} \js\ $\tilde{\fT}$ of rank $m + \lfloor m/2\rfloor +1$
with the following set of basic relations:
$$\tilde{\cR} = \{C_i\cup C_i^\top\,|\,i\in\Z_m\}\cup \{S_i\,|\,i\in\Z_m\}.$$ 

We will show how to get a new Jordan scheme from this one by a certain switching of color graphs. To define the switching we partition the point set $\Omega$ into a union of two disjoint subsets $\Omega= \Omega_1\cup \Omega_2$ where 
$\Omega_1$ is an arbitrary fiber of $E$ and $\Omega_2:=\Omega\setminus \Omega_1$. 
Then every relation $S_i,i=0,...,m$ splits into two disjoint subsets: the edges between the parts $S_i^b:=S_i \cap (\Omega_1\times \Omega_2\cup \Omega_2\times \Omega_1)$ and the edges within the parts $S_i^w:=S_i\cap (\Omega_2\times \Omega_2)$ (note that $S_i\cap (\Omega_1\times \Omega_1)=\emptyset$).

It is worth to note that $(\Omega,S_i^b)$ is a bipartite graph with $|S_i^b(\omega)|=n$ if $\omega\in \Omega_1$ and $|S_i^b(\omega)|=1$ if $\omega\in \Omega_2$. The graph $(\Omega,S_i^b)$ is a disjoint union of $m$ copies of $K_{1,n}$. In particular, the sets $S_i^b(\omega)$ and $S_i^b(\omega')$ are disjoint whenever $\omega'\neq \omega\in\Omega_1$.  The graph $(\Omega_2,S_i^w)$ is a regular graph of order $mn$ and valency $n-1$.

The main result of this section is the following
\begin{theorem}\label{310319b} The relations $D_i:=C_i\cup C_{-i},0\leq i\leq m/2, T_i:=S_i^b\cup S_{-i}^w,i\in\Z_m$  form a proper \js   \ denoted as $\fJ$.
\end{theorem}

To prove the theorem we will write the matrices of $M_\Omega({\mathbb R})$ 
as $2\times 2$ block matrices $A = \mtrx{A_{11}}{A_{12}}{A_{21}}{A_{22}}$ where the $(i,j)$-block corresponds to $\Omega_i\times \Omega_j$-part of the matrix $A$. To ease notation we abbreviate $J_\Omega$ as $J$ and denote its $(i,j)$-block as $J_{ij}$.

In a block form the matrices of the original scheme look as follows:
$$
\und{C}_i = \mtrx{({\und{C}_i})_{11}}{O}{O}{(\und{C}_i)_{22}},\ 
\und{S}_j = \mtrx{O}{(\und{S}_j)_{12}}{(\und{S}_j)_{21}}{(\und{S}_j)_{22}}.
$$

We introduce the relation $F:=(S_0)_{21}(S_0)_{12}=(S_0)_{21}(S_0)_{21}^\top$. Since 
$(S_0)_{21} = S_0\cap (\Omega_2\times\Omega_1)$ is a surjective function from $\Omega_2$ onto $\Omega_1$,
the relation $F$ is an equivalence relation on $\Omega_2$ defined
as follows \footnote{$F$ is the kernel of the function $(S_0)_{21}$.}
$$\forall_{\omega,\omega'\in\Omega_2}\ (\omega,\omega')\in F\iff  (S_0)_{21}(\omega)=(S_0)_{21}(\omega')\iff S_0(\omega)\cap\Omega_1=S_0(\omega')\cap\Omega_1.
$$
In the matrix language the above equality transfers to $\und{F}=(\und{S}_0)_{21} (\und{S}_0)_{12}$. Since $(S_0)_{21}$ is 
an $n$-to-$1$ surjective function, the equivalence $F$ has $m$ classes of cardinality $n$.

                                                                                                                                                                                                                                                                                                                                                                                                                                                                                                                                                                                                                                                                                                                                                                                                                                              The statement below describes some matrix products that we need.
\begin{prop}\label{010419c} For any $i,j\in\Z_m$ it holds that
$$
\begin{array}{rcl}
(\und{S}_i)_{12} (\und{S}_j)_{21} & = & n (\und{C}_{i-j})_{11};\\
(\und{S}_i)_{12} (\und{S}_j)_{22} & = & \frac{n-1}{m} J_{12};\\
(\und{S}_i)_{22} (\und{S}_i)_{21} & =  & \frac{n-1}{m} J_{21};\\
(\und{S}_i)_{21} (\und{S}_j)_{12} & = &  \und{F}\cdot (\und{C}_{i-j})_{22};\\
(\und{S}_i)_{22} (\und{S}_j)_{22} & = &  \frac{n-1}{m}(J_{22} -\und{E}_{22}) + n (\und{C}_{i-j})_{22} -\und{F} \cdot (\und{C}_{i-j})_{22}.\\
\end{array}
$$
\end{prop}
\begin{proof} It follows from the formulae~\eqref{eq:310319a} that
$\und{S}_i\cdot \und{S}_j = n \und{C}_{i-j} + \frac{n-1}{m} (J - \und{E}).$ Writing this equality in a block-matrix form yields us
$$
\begin{array}{rcl}
(\und{S}_i)_{12} (\und{S}_j)_{21} & = & n (\und{C}_{i-j})_{11};\\
(\und{S}_i)_{12} (\und{S}_j)_{22} & = & \frac{n-1}{m} J_{12};\\
(\und{S}_i)_{22} (\und{S}_i)_{21} & =  & \frac{n-1}{m} J_{21};\\
(\und{S}_i)_{21}\cdot(\und{S}_j)_{12} + (\und{S}_i)_{22}\cdot (\und{S}_j)_{22} & = & \frac{n-1}{m}(J_{22} - \und{E}_{22}) + n (\und{C}_{i-j})_{22}.
\end{array}
$$
This proves the first three rows of our statement.

Since the fifth row is a direct consequence of the fourth one, it
remains to prove the fourth row only.  

It follows from \eqref{eq:310319a}
that $(\und{C}_k)_{22} \cdot (\und{S}_0)_{21} = (\und{S}_k)_{21} = (\und{S}_0)_{21}\cdot (\und{C}_{-k})_{11}$ and $(\und{S}_0)_{12}\cdot(\und{C}_k)_{22}  = (\und{S}_{-k})_{12} = (\und{C}_{-k})_{11}\cdot (\und{S}_0)_{12} $. Therefore, 
$$
(\und{S}_i)_{21}\cdot (\und{S}_j)_{12} = (\und{S}_0)_{21} \cdot (\und{C}_{-i})_{11}(\und{S}_0)_{12} (\und{C}_{-j})_{22}
 = (\und{S}_0)_{21}\cdot (\und{S}_0)_{12} \cdot (\und{C}_{i})_{22}(\und{C}_{-j})_{22}  
= \und{F} \cdot (\und{C}_{i-j})_{22}.
$$
\end{proof}

{\bf Proof} of Theorem~\ref{310319b}. 
\textcolor{black}{We first prove that $\fJ$ is a Jordan scheme and then that it is proper.}

Both parts of the proof are purely computational. In what follows $\cB$ stands for the linear span of the matrices $\und{D_i},\und{T_j}$.

To ease the notation we omit underlining in this proof, i.e. the relations will be identified with their adjacency matrices.

To make calculations easier we introduce $\tilde{D}_i:=C_i+C_{-i}$.
Note $\tilde{D}_i = D_i$ unless $i=0$ or $i=m/2$. In these two cases $\tilde{D}_0 = 2 D_0, \tilde{D}_{m/2} = 2 D_{m/2}$ (in the latter case $m$ should be even).

It is easy to check that $\tilde{D}_i\star \tilde{D}_j = \tilde{D}_{i+j} + \tilde{D}_{i-j}\in\cB$.

Now we check that $\tilde{D}_i\star T_j\in\cB$. 
$$
2\tilde{D}_i\star T_j = 2(C_i+C_{-i})\star T_j = (C_i+C_{-i})\cdot T_j + T_j\cdot(C_i+C_{-i}) =
$$
$$ 
\mtrx{(C_i)_{11} + (C_{-i})_{11}}{O}{O}{(C_i)_{22} + (C_{-i})_{22}}
\mtrx{O}{(S_j)_{12}}{(S_j)_{21}}{(S_{-j})_{22}} +
$$
$$
\mtrx{O}{(S_j)_{12}}{(S_j)_{21}}{(S_{-j})_{22}} 
\mtrx{(C_i)_{11} + (C_{-i})_{11}}{O}{O}{(C_i)_{22} + (C_{-i})_{22}} =
$$
$$
\mtrx{O}{(S_{j+i})_{12} + (S_{j-i})_{12}}{(S_{j+i})_{21} + 
(S_{j-i})_{21}}{(S_{-j+i})_{22} + (S_{-j-i})_{22}} +
\mtrx{O}{(S_{j-i})_{12} + (S_{j+i})_{12}}{(S_{j-i})_{21} + 
(S_{j+i})_{21}}{(S_{-j-i})_{22} + (S_{-j+i})_{22}} =
$$
$$
2\mtrx{O}{(S_{j+i})_{12} + (S_{j-i})_{12}}{(S_{j+i})_{21} + 
(S_{j-i})_{21}}{(S_{-j+i})_{22} + (S_{-j-i})_{22}} = 2T_{j+i} + 2T_{j-i}\implies  \tilde{D}_i\star T_j =  T_{j+i} + T_{j-i}\in\cB.
$$

Now we compute $T_i\star T_j$. We start with $T_i \cdot T_j$:
$$
T_i\cdot T_j = \mtrx{O}{(S_i)_{12}}{(S_i)_{21}}{(S_{-i})_{22}} 
\mtrx{O}{(S_j)_{12}}{(S_j)_{21}}{(S_{-j})_{22}} = 
\mtrx{(S_i)_{12}(S_j)_{21}}{(S_i)_{12}(S_{-j})_{22}}{(S_{-i})_{22}(S_j)_{21}}{(S_i)_{21}(S_j)_{12} + (S_{-i})_{22}(S_{-j})_{22}}.
$$
To compute the latter matrix we use formulae of Proposition~\ref{010419c}:
$$
\begin{array}{rcl}
(S_i)_{12}(S_j)_{21} & = & n (C_{i-j})_{11};\\
(S_i)_{12}(S_{-j})_{22} & = & \frac{n-1}{m} J_{12};\\
(S_{-i})_{22}(S_j)_{21} & = & \frac{n-1}{m} J_{21} 
\end{array}
$$
and
$$
(S_i)_{21}(S_j)_{12} + (S_{-i})_{22}(S_{-j})_{22} = 
F\cdot (C_{i-j})_{22} +
\frac{n-1}{m} (J_{22} - E_{22}) + n (C_{-i+j})_{22} - F\cdot (C_{-i+j})_{22}. 
$$

Thus 
\begin{equation}\label{050419a}
T_i\cdot T_j = 
\mtrx{n (C_{i-j})_{11}}{\frac{n-1}{m} J_{12}}{\frac{n-1}{m} J_{21}}
{F\cdot (C_{i-j})_{22} +
\frac{n-1}{m} (J_{22} - E_{22}) + n (C_{-i+j})_{22} - F\cdot (C_{-i+j})_{22}}.
\end{equation}
Swapping $i$ with $j$ we obtain
$$
T_j\cdot T_i = 
\mtrx{n (C_{j-i})_{11}}{\frac{n-1}{m} J_{12}}{\frac{n-1}{m} J_{21}}
{F\cdot (C_{j-i})_{22} +
\frac{n-1}{m} (J_{22} - E_{22}) + n (C_{-j+i})_{22} - F\cdot (C_{-j+i})_{22}}.
$$
Adding the above formulae yields us
$$
T_i\star T_j = \frac{1}{2}
\mtrx{n (C_{i-j})_{11} + n (C_{j-i})_{11}}{2\frac{n-1}{m} J_{12}}{2\frac{n-1}{m} J_{21}}
{2\frac{n-1}{m} (J_{22} - E_{22}) + n (C_{-j+i})_{22} + n (C_{j-i})_{22}} = \frac{n}{2} D_{i-j} + \frac{n-1}{m}(J - E).
$$

It remains to show that the above Jordan scheme is proper. By Proposition~\ref{081219a} this is equivalent to $\cB\neq\widetilde{WL(\cB)}$. To \textcolor{black}{justify this} we calculate some elements of the coherent closure $WL(\cB)$.

The product $T_i\cdot T_j$ belongs to $WL(\cB)$. It follows from~\eqref{050419a} that 
$$
(T_i\cdot T_j)\circ E =
\mtrx{n (C_{i-j})_{11}}{O}{O}
{F\cdot (C_{i-j})_{22}\circ E_{22} +
 n (C_{-i+j})_{22} - F\cdot (C_{-i+j})_{22}\circ E_{22}}\in WL(\cB).
$$

Since $(C_{-i+j})_{22}$ is a permutation matrix, we can write\footnote{Here we use the identity $X\cdot P\circ Y\cdot P = (X\circ Y) P$ which holds for any permutation matrix $P$ and arbitrary matrices $X,Y$ of appropriate orders.}

$$
F\cdot (C_{-i+j})_{22}\circ E_{22} = F\cdot (C_{-i+j})_{22}\circ E_{22}\cdot (C_{-i+j})_{22} 
 = 
\left(
F\circ E_{22}\right) 
\cdot (C_{-i+j})_{22}.$$
Taking into account that $E_{22}\circ F = I_{22}$ we obtain that $(F\cdot (C_{-i+j})_{22})\circ E_{22} = (C_{-i+j})_{22}$. Analogously, 
$F\cdot (C_{i-j})_{22}\circ E_{22} = (C_{i-j})_{22}$. Therefore 
$$
(T_i\cdot T_j)\circ E =
\mtrx{n (C_{i-j})_{11}}{O}{O}
 {(C_{i-j})_{22} + (n-1) (C_{-i+j})_{22}}\in WL(\cB).
$$
By the Schur-Wielandt principle we conclude that $WL(\cB)$ contains the following matrices
$$
\mtrx{(C_{i-j})_{11}}{O}{O}{O},
\mtrx{O}{O}{O}{(C_{i-j})_{22}},
\mtrx{O}{O}{O}{(C_{-i+j})_{22}}.
$$
Therefore 
$$
A:=\mtrx{(C_{i-j})_{11}}{O}{O}{O}+{\mtrx{(C_{i-j})_{11}}{O}{O}{O}}^\top\in\widetilde{WL(\cB)}.
$$
Together with $A\not\in\cB$ we obtain the required inequality $\widetilde{WL(\cB)}\neq \cB$.\hfill $\square$

It follows from the above proof that $\fJ$ has the following multiplication table
\begin{equation}\label{310819a}
\begin{array}{rcl}
\tilde{D}_i\star \tilde{D}_j & = & \tilde{D}_{i+j} + \tilde{D}_{i-j};\\
\tilde{D}_i\star T_j & = & T_{j+i} + T_{j-i};\\
T_i\star T_j & = & \frac{n}{2} \tilde{D}_{i-j} + \frac{n-1}{m} (J - E).
\end{array}
\end{equation}
One can easily check that the Jordan schemes $\tilde{\fT}$ and $\fJ$ are algebraically isomorphic.

\textcolor{black}{\noindent
{\bf Remark.} The smallest example of a proper Jordan scheme $\fJ$ appears when $m=3$ and $n=4$. It has order $15$ and rank $5$. The existence of this example stems from \cite{KRW}. An interested reader is referred to the essay \cite{KMR}, Sections 4-12, where this structure, denoted by $J_{15}$, is treated on more than 40 pages with an excessive attention to all related details.
}
\section{Appendix}\label{appendix}
\textcolor{black}{The proof of the statement below relies on a standard knowledge of the theory of commutative associative algebras (see, for example, \cite{DK}).}
\begin{prop}\label{120619a} Let $\cJ\subseteq \sym_\Omega(\mathbb{R})$ be a $\star$-subalgebra. If $(\cJ,\star)$ is associative, then $x\star y = x\cdot y$ and $\cJ$ is $\cdot$-commutative.
\end{prop}
\begin{proof}
The algebra $(\cJ,\star)$ is commutative and associative. It does not contain nilpotent elements, because a symmetric real matrix cannot be nilpotent. Therefore $(\cJ,\star)$ is semisimple. Let $E_1,...,E_d$ be a complete set of primitive idempotents. Then $E_i\star\cJ$ is a field isomorphic either to $\mathbb{C}$ or $\mathbb{R}$. 
Since $\cJ$ is formally real, the case of $E_i\star\cJ\cong\mathbb{C}$ is impossible. Therefore $E_i\star\cJ\cong \mathbb{R}$ for each $i=1,...,d$. It follows then that $E_1,...,E_d$ is a basis of $\cJ$. 

Since $\{E_i\}_{i=1}^d$ are primitive idempotents of $(\cJ,\star)$, they are pairwise orthogonal: $E_i\star E_j=\delta_{i,j} E_i$. This implies that $E_i\cdot E_i=E_i$ and $E_i\cdot E_j + E_j\cdot E_i =0$ whenever $i\neq j$. It remains to show that $i\neq j\implies E_i\cdot E_j =0$. Since $E_i$ is an idempotent matrix it has eigenvalues $0,1$ and the vector space $\mathbb{R}^\Omega$ has a direct sum decomposition:
$\mathbb{R}^\Omega = V_1\oplus V_0$ where $V_1,V_0$ are the corresponding eigenspaces of $E_i$. Consider the product $E_j E_i$ with $i\neq j$. Clearly that $E_j E_i V_0 = 0$. Pick an arbitrary $v\in V_1$. It follows from $E_j E_i + E_i E_j=0$  that $E_j v = - E_i (E_j v)$. Since $-1$ is not an eigenvalue of $E_i$, we conclude that $E_j v=0$. Therefore $E_j E_i v = 0$ implying $E_j E_i V_1 = 0$. Thus $E_i E_j =0$ whenever $i\neq j$.
Now using the basis $E_1,..,E_d$ one can finish the proof of the statement.
\end{proof}

\end{document}